\documentclass[a4paper,12pt,oneside,headsepline]{scrartcl}

\usepackage{ifdraft}

\usepackage[utf8]{inputenc}
\usepackage[T1]{fontenc}

\usepackage{lmodern}
\usepackage{fourier}
\usepackage{amssymb, amsfonts}

\usepackage{mathrsfs} %
\usepackage{dsfont}
\usepackage[usenames,dvipsnames]{xcolor}%
\usepackage{lettrine}
\usepackage{url}

\usepackage{stmaryrd}

\usepackage[english]{babel}

\usepackage[draft=false]{scrlayer-scrpage}

\usepackage{wrapfig}
\usepackage{footmisc}
\usepackage{needspace}

\usepackage{amsmath}
\usepackage{nccmath}
\usepackage[thmmarks, amsmath, amsthm]{ntheorem}

\usepackage[style=numeric-comp,giveninits,url=false,sortcites=true,maxbibnames=99]{biblatex}
\bibliography{bibliography.bib}

\usepackage[pdftex,final,
             pdfborder={0 0 0}, colorlinks=true,
             linkcolor=BrickRed, citecolor=ForestGreen, urlcolor=RoyalBlue]{hyperref}

\ifdraft{%
\usepackage[colorinlistoftodos]{todonotes}
\usepackage{lineno}
}{}

\usepackage{booktabs}
\usepackage[inline]{enumitem}
\usepackage{float}
\usepackage{graphicx}
\usepackage{multicol}
\usepackage{scrtime}
\usepackage{tikz}
\usetikzlibrary{matrix,arrows,positioning}

\ifdraft{\date{\today}}{\date{}}

\ifdraft{%
\linenumbers
\newcommand*\patchAmsMathEnvironmentForLineno[1]{%
  \expandafter\let\csname old#1\expandafter\endcsname\csname #1\endcsname
  \expandafter\let\csname oldend#1\expandafter\endcsname\csname end#1\endcsname
  \renewenvironment{#1}%
     {\linenomath\csname old#1\endcsname}%
     {\csname oldend#1\endcsname\endlinenomath}}%
\newcommand*\patchBothAmsMathEnvironmentsForLineno[1]{%
  \patchAmsMathEnvironmentForLineno{#1}%
  \patchAmsMathEnvironmentForLineno{#1*}}%
\AtBeginDocument{%
\patchBothAmsMathEnvironmentsForLineno{equation}%
\patchBothAmsMathEnvironmentsForLineno{align}%
\patchBothAmsMathEnvironmentsForLineno{flalign}%
\patchBothAmsMathEnvironmentsForLineno{alignat}%
\patchBothAmsMathEnvironmentsForLineno{gather}%
\patchBothAmsMathEnvironmentsForLineno{multline}%
}}

\KOMAoptions{DIV=10}

\clearscrheadings
\automark[section]{subsection}

\renewcommand{\subsectionmark}[1]{}

\lohead{{\small \headertitle}}
\rohead{{\small \headerauthors}}

\RedeclareSectionCommand[%
  font=\Large\sffamily\bfseries,%
  beforeskip=1\baselineskip,%
  afterskip=0.5\baselineskip,%
  indent=0em%
  ]{section}

\RedeclareSectionCommands[%
  font=\normalfont\bfseries,%
  afterskip=-1em%
  ]{subsection,subsubsection}

\RedeclareSectionCommands[%
  font=\normalfont\itshape,%
  afterskip=-1em,%
  indent=0pt,
  ]{paragraph}

\AtEveryBibitem{\clearfield{isbn}}
\AtEveryBibitem{\clearlist{language}}
\AtEveryBibitem{\clearfield{pages}}

\newenvironment{enumeratearabic}{
\begin{enumerate}[label=(\arabic*), leftmargin=0pt,labelindent=2em,itemindent=!]
}{
\end{enumerate}
}

\newenvironment{enumerateroman}{
\begin{enumerate}[label=(\roman*), leftmargin=0pt,labelindent=2em,itemindent=!]
}{
\end{enumerate}
}

\newenvironment{enumeratearabic*}{
\begin{enumerate*}[label=(\arabic*)] %
}{
\end{enumerate*}
}

\newenvironment{enumerateroman*}{
\begin{enumerate*}[label=(\roman*)] %
}{
\end{enumerate*}
}

\numberwithin{equation}{section}

\theoremnumbering{arabic}
\newtheorem{theoremcounter}{theoremcounter}[section]
\theoremnumbering{Alph}
\newtheorem{maintheoremcounter}{maintheoremcounter}

\theoremstyle{plain}

\newtheorem{corollary}[theoremcounter]{Corollary}
\newtheorem{lemma}[theoremcounter]{Lemma}

\newtheorem{proposition}[theoremcounter]{Proposition}
\newtheorem{theorem}[theoremcounter]{Theorem}

\theoremstyle{plain}

\newtheorem{maincorollary}[maintheoremcounter]{Corollary}
\newtheorem{maintheorem}[maintheoremcounter]{Theorem}

\theoremstyle{definition}

\newtheorem{example}[theoremcounter]{Example}

\newtheorem{mainexample}[maintheoremcounter]{Example}

\theoremstyle{remark}

\newtheorem{remark}[theoremcounter]{Remark}

\theoremstyle{nonumberremark}

\newtheorem{mainremark}{Remark}
\newtheorem{remarkcomputation}{Computation}

\newenvironment{mainremarkenumerate}
{%
\mainremark
\enumeratearabic
}{%
\endenumeratearabic
\endmainremark
}%

\let\cal\undefined

\ifdraft{
 \newcommand{\texpdf}[2]{#1}
}{
 \newcommand{\texpdf}[2]{\texorpdfstring{#1}{#2}}
}

\newcommand{\tx}{\ensuremath{\text}}

\newcommand{\tbf}{\bfseries}

\newcommand{\thdash}{\nbd th}

\newcommand{\nbd}{\nobreakdash-\hspace{0pt}}

\newcommand{\bboard}{\ensuremath{\mathbb}}
\newcommand{\cal}{\ensuremath{\mathcal}}
\renewcommand{\frak}{\ensuremath{\mathfrak}}

\newcommand{\bbA}{\ensuremath{\bboard A}}

\newcommand{\bbF}{\ensuremath{\bboard F}}

\newcommand{\bbP}{\ensuremath{\bboard P}}

\newcommand{\cO}{\ensuremath{\cal{O}}}

\newcommand{\frake}{\ensuremath{\frak{e}}}

\newcommand{\fraku}{\ensuremath{\frak{u}}}

\newcommand{\rmc}{\ensuremath{\mathrm{c}}}

\newcommand{\rmf}{\ensuremath{\mathrm{f}}}

\newcommand{\rml}{\ensuremath{\mathrm{l}}}

\newcommand{\rms}{\ensuremath{\mathrm{s}}}
\newcommand{\rmt}{\ensuremath{\mathrm{t}}}

\newcommand{\rmE}{\ensuremath{\mathrm{E}}}
\newcommand{\rmF}{\ensuremath{\mathrm{F}}}

\newcommand{\rmL}{\ensuremath{\mathrm{L}}}
\newcommand{\rmM}{\ensuremath{\mathrm{M}}}

\newcommand{\rmR}{\ensuremath{\mathrm{R}}}
\newcommand{\rmS}{\ensuremath{\mathrm{S}}}
\newcommand{\rmT}{\ensuremath{\mathrm{T}}}
\newcommand{\rmU}{\ensuremath{\mathrm{U}}}
\newcommand{\rmV}{\ensuremath{\mathrm{V}}}

\newcommand{\wtd}{\widetilde}
\newcommand{\ov}{\overline}

\newcommand{\llbrkt}{\llbracket}
\newcommand{\rrbrkt}{\rrbracket}

\newcommand*{\longhookrightarrow}{\ensuremath{\lhook\joinrel\relbar\joinrel\rightarrow}}

\newcommand{\ra}{\ensuremath{\rightarrow}}

\newcommand{\thra}{\ensuremath{\twoheadrightarrow}}

\newcommand{\lra}{\ensuremath{\longrightarrow}}
\newcommand{\lhra}{\ensuremath{\longhookrightarrow}}

\newcommand{\mto}{\ensuremath{\mapsto}}
\newcommand{\lmto}{\ensuremath{\longmapsto}}

\newcommand{\amid}{\ensuremath{\mathop{\mid}}}

\newcommand{\ZZ}{\ensuremath{\mathbb{Z}}}
\newcommand{\QQ}{\ensuremath{\mathbb{Q}}}
\newcommand{\RR}{\ensuremath{\mathbb{R}}}
\newcommand{\CC}{\ensuremath{\mathbb{C}}}

\renewcommand{\Im}{\ensuremath{\mathrm{Im}}}

\newcommand{\isdiv}{\amid}
\newcommand{\nisdiv}{\ensuremath{\mathop{\nmid}}}

\renewcommand{\pmod}[1]{\ensuremath{\;(\mathrm{mod}\, #1)}}

\newenvironment{psmatrix}{\left(\begin{smallmatrix}}{\end{smallmatrix}\right)}

\newcommand{\Mat}[1]{\ensuremath{\mathrm{Mat}_{#1}}}
\newcommand{\GL}[1]{\ensuremath{\mathrm{GL}_{#1}}}

\newcommand{\SL}[1]{\ensuremath{\mathrm{SL}_{#1}}}
\newcommand{\Mp}[1]{\ensuremath{\mathrm{Mp}_{#1}}}

\newcommand{\U}[1]{\ensuremath{\mathrm{U}_{#1}}}

\newcommand{\T}{\ensuremath{\rmt}}

\renewcommand{\det}{\ensuremath{\mathrm{det}}}

\renewcommand{\ker}{\ensuremath{\mathrm{ker}}}

\newcommand{\HS}{\mathbb{H}}

\newcommand{\lcm}{\ensuremath{\mathrm{lcm}}}
\newcommand{\isexdiv}{\ensuremath{\mathop{\|}}}

\newcommand{\ovellker}{\ensuremath{\ov{\ker}_{\rmF\rmE\ell}}}
\newcommand{\ellker}{\ensuremath{\ker_{\rmF\rmE\ell}}}

\newcommand{\om}{\ensuremath{\omega}}
\newcommand{\ga}{\ensuremath{\gamma}}
\newcommand{\Ga}{\ensuremath{\Gamma}}
\newcommand{\tdGa}{\ensuremath{\wtd{\Gamma}}}

\newcommand{\GMp}[1]{\ensuremath{\mathrm{GMp}_{#1}}}

\newcommand{\Ind}{\ensuremath{\mathrm{Ind}}}

\newcommand{\rmfe}{\ensuremath{\mathrm{fe}}}
\newcommand{\rmFE}{\ensuremath{\mathrm{FE}}}

\newcommand{\OK}{\ensuremath{\cO_K}}
\newcommand{\OKell}{\ensuremath{\cO_{K,\ell}}}
\newcommand{\Fell}{\ensuremath{\rmF_{K,\ell}}}

\newcommand{\rmTcl}{\ensuremath{\rmT^{\rmc\rml}}}

\renewcommand{\longmapsfrom}{\mathrel{\reflectbox{\ensuremath{\longmapsto}}}}

\newcommand{\headertitle}{{\normalfont%
  Relations among Ramanujan-Type Congruences I
}}
\newcommand{\headerauthors}{%
  M.~Raum%
}
\title{%
  Relations among\\Ramanujan-Type Congruences I
}
\subtitle{%
  Ramanujan-type Congruences in Integral Weights
}
\author{%
Martin Raum%
\thanks{The author was partially supported by Vetenskapsr\aa det Grant~2015-04139 and~2019-03551.}%
}
\date{}%

\begin{document}

\thispagestyle{scrplain}
\begingroup
\deffootnote[1em]{1.5em}{1em}{\thefootnotemark}
\maketitle
\endgroup

{\small
\noindent
{\tbf Abstract:}
We prove that Ramanujan-type congruences for integral weight modular forms away from the level and the congruence prime are equivalent to specific congruences for Hecke eigenvalues. In particular, we show that Ramanujan-type congruences are preserved by the action of the shallow Hecke algebra. More generally, we show for weakly holomorphic modular forms of integral weight, that Ra\-ma\-nu\-jan-type congruences naturally occur for shifts in the union of two square-classes as opposed to single square-classes that appear in the literature on the partition function. We also rule out the possibility of square-free periods, whose scarcity in the case of the partition function was investigated recently. We complement our obstructions on maximal Ramanujan-type congruences with several existence statements. Our results are based on a framework that leverages classical results on integral models of modular curves via modular representation theory, and applies to congruences of all weakly holomorphic modular forms. Steinberg representations govern all maximal Ramanujan-type congruences for integral weights. We discern the scope of our framework in the case of half-integral weights through example calculations.
\\[.35em]
\textsf{\textbf{%
  Hecke and $\rmU_p$-congruences%
}}%
\hspace{0.3em}{\tiny$\blacksquare$}\hspace{0.3em}%
\textsf{\textbf{%
  Fourier coefficients of
  weakly holomorphic modular forms%
}}%
\hspace{0.3em}{\tiny$\blacksquare$}\hspace{0.3em}%
\textsf{\textbf{%
  modular representation theory%
}}%
\hspace{0.3em}{\tiny$\blacksquare$}\hspace{0.3em}%
\textsf{\textbf{%
  Steinberg representations%
}}
\\[0.15em]
\noindent
\textsf{\textbf{%
  MSC Primary:
  11F33%
}}
\hspace{0.3em}{\tiny$\blacksquare$}\hspace{0.3em}%
\textsf{\textbf{%
  MSC Secondary:
  11F30%
}}
}
\\[-.5\baselineskip]

\newpage
\setcounter{tocdepth}{2}%
\renewcommand{\contentsname}{\vspace{-1.0\baselineskip}}%
\tableofcontents
\vspace{1.5\baselineskip}

\Needspace*{4em}
\addcontentsline{toc}{section}{Introduction}
\markright{Introduction}
\lettrine[lines=2,nindent=.2em]{\tbf C}{ongruences} for Fourier coefficients of modular forms on arithmetic progressions come in two distinct kinds. There are Ramanujan-type congruences which are especially important in combinatorics, and the special case of~$\rmU_p$\nbd congruences, which enjoy a much more arithmetic geometric theory via their connection to slopes of~$p$\nbd adic modular forms. In this work, we characterize Ramanujan-type congruences for integral weight modular forms in terms of congruences for Hecke eigenvalues and make them thus equally accessible to arithmetic geometric tools.

Ramanujan in the 1920ies found astounding divisibility patterns for the partition function~\cite{ramanujan-1920}:
\begin{gather*}
  p(5 n + 4) \;\equiv\; 0 \;\pmod{5}
\tx{,}\quad
  p(7 n + 5) \;\equiv\; 0 \;\pmod{7}
\tx{,}\quad
  p(11 n + 6) \;\equiv\; 0 \;\pmod{11}
\tx{.}
\end{gather*}
When discovered in the 1920ies, they surprised number theorists as they provided regular divisibility for a function whose mere calculation for small arguments was at the edge of the feasible. Only decades later, it was uncovered that they fit into a larger theme: Given a weakly holomorphic modular form~$f$, a prime~$\ell$, and integers~$M > 0$ and~$\beta$, we say that~$f$ has a Ramanujan-type congruences modulo~$\ell$ on the arithmetic progression~$M \ZZ + \beta$ if its Fourier coefficients~$c(f;\,n)$ satisfy
\begin{gather*}
  \forall n \in \ZZ :\,
  c(f;\,M n + \beta) \equiv 0 \;\pmod{\ell}
\tx{.}
\end{gather*}
We say that this congruence is maximal, if~$f$ does not satisfy a Ramanujan-type congruence modulo~$\ell$ on~$M' \ZZ + \beta$ for any proper divisor~$M'$ of~$M$. Ramanujan's congruences correspond to the case of~$f = 1 / \eta$, the inverse of the De\-de\-kind~$\eta$\nbd function,~$\ell = M \in \{5, 7, 11\}$, and~$\beta \in M \ZZ_M \cap (\ZZ-\frac{1}{24})$. Since~$M$ is a power of~$\ell$, they are examples of the more special~$\rmU_p$\nbd congruences.

Instances of Ramanujan-type congruences for~$1 \slash \eta$ beyond~$\rmU_p$\nbd congruences were explored by Atkin and collaborators~\cite{atkin-obrien-1967,atkin-1968} in the 1960ies, using both novel computer facilities for experiments and the theory of Hecke operators for purely mathematical investigations. He found infinite families of Ramanujan-type congruences, but remained limited to finitely many~$\ell$.

Eventually in the 2000s, Ono~\cite{ono-2000} and Ahlgren-Ono~\cite{ahlgren-ono-2001} provided Ramanujan-type congruences for~$1 \slash \eta$ for all~$\ell > 3$. Like Atkin, they employed Hecke operators to derive their results. The omnipresence of Hecke operators in work on Ramanujan-type congruences suggests a tight link between these and congruences for the action of Hecke operators. This link is manifest in the next two theorems.

We write~$\rmTcl_p$ for the~$p$\thdash\ Hecke operator acting on modular forms.
\begin{maintheorem}
\label{mainthm:hecke-stable}
Fix an odd prime~$\ell$, an integral weight~$k$, and a Dirichlet character~$\chi$ modulo~$N$. Fix a positive integer~$M$ and an integer~$\beta$. Then the space of cusp forms~$f$ of weight~$k$ for~$\chi$ with~$\ell$\nbd integral Fourier coefficients~$c(f;\, n)$ that satisfy the Ramanujan-type congruence
\begin{gather*}
  \forall n \in \ZZ \,:\,
  c(f;\,M n + \beta) \equiv 0 \;\pmod{\ell}
\end{gather*}
is stable under the action of the Hecke operators~$\rmTcl_p$, $p \nisdiv \gcd(M, \ell N)$.
\end{maintheorem}
\begin{mainremarkenumerate}
\item
A statement subsuming the case of algebraic Fourier coefficients is given in Theorem~\ref{thm:hecke-module-ramanujan-type-congruences}.

\item
The restriction that~$\ell$ is odd in this introduction arises from the condition~$\ell \nisdiv 2$ in Theorem~\ref{thm:structure-of-ramanujan-type-congruences}~\ref{it:thm:structure-of-ramanujan-type-congruences:ramanujan-type-congruence-with-gap}. In other words, if~$\ell$ is even, the conclusions in our main tools, Theorems~\ref{mainthm:structure-of-ramanujan-type-congruences} and~\ref{thm:structure-of-ramanujan-type-congruences}, fail and we are thus unable to deduce Theorems~\ref{mainthm:hecke-stable} and~\ref{mainthm:ramanujan-type-implies-hecke} and Corollary~\ref{maincor:ramanujan-type-has-density}. However, modular forms of integral weight are congruent modulo~$2$ to modular forms of half-integral weight. We can apply the very different techniques that we develop in the second part of this series in the latter setting.
\end{mainremarkenumerate}

We say that a modular forms is a generalized Hecke eigenform modulo~$\ell$ for~$\rmTcl_p$ if there is an $\ell$-integral scalar~$\lambda_p$ and a nonnegative integer~$d$ such that we have the congruence~$f | (\rmTcl_p - \lambda_p)^{d+1} \equiv 0 \,\pmod{\ell}$. We refer to~$\lambda_p$ as its eigenvalue, and to the minimal possible nonnegative~$d$ as its depth. Theorem~\ref{mainthm:hecke-stable} implies that if a linear combination~$f = \sum_\lambda f_\lambda$ satisfies a Ramanujan-type congruence for generalized Hecke eigenforms~$f_\lambda$ for~$\rmTcl_p$, $p \nisdiv \gcd(M,\ell N)$, of eigenvalues~$\lambda$ that are pairwise different modulo~$\ell$, then each~$f_\lambda$ satisfies the same Ramanujan-type congruence. In particular, the assumption in Theorem~\ref{mainthm:ramanujan-type-implies-hecke} that~$f$ is a generalized Hecke eigenform does not limit its scope.
\begin{maintheorem}
\label{mainthm:ramanujan-type-implies-hecke}
Fix an odd prime~$\ell$, and let~$f$ be a cusp form of integral weight~$k$ for a Dirichlet character~$\chi$ modulo~$N$ with~$\ell$\nbd integral Fourier coefficients~$c(f;\,n)$. Fix a power~$M$ of a prime~$p$, $p \nisdiv \ell N$, and an integer~$\beta$. Assume that~$f$ is a generalized Hecke eigenform modulo~$\ell$ for~$\rmTcl_p$ of eigenvalue~$\lambda_p$ of depth~$d$.

Provided that~$f \not\equiv 0 \,\pmod{\ell}$, the following are equivalent:
\begin{enumeratearabic}
\item
We have the Ramanujan-type congruence
\begin{gather*}
  \forall n \in \ZZ \,:\,
  c(f;\,M n + \beta) \equiv 0 \;\pmod{\ell}
\tx{.}
\end{gather*}

\item
For every\/~$0 \le t \le d$ the coefficient of exponent\/~$\gcd(M, \beta)$ in the following formal power series vanishes modulo~$\ell$:
\begin{gather*}
  X^t\, \big( 1 - \lambda_p X + \chi(p) p^{k-1} X^2 \big)^{-t-1}
\tx{.}
\end{gather*}
\end{enumeratearabic}
\end{maintheorem}
\begin{mainremarkenumerate}
\item
A statement subsuming the case of algebraic Fourier coefficients is given in Theorem~\ref{thm:ramanujan-type-implies-hecke}.

\item
The assumption that~$M$ is a prime power hints at a substantial complication in the theory: The structure of the $\ell$\nbd adic completion of the Hecke algebra in the algebraic-geometric sense is not known in general. Recent work, by for example Bella\"{i}che--Khare~\cite{bellaiche-khare-2015} or Wake--Wang-Erickson~\cite{wake-wang-erickson-2021}, illustrates the subtleties that can arise. To meet the assumption on~$M$, we can replace~$f$ with a maximal Ramanujan-type congruence on~$M \ZZ + \beta$ by the modular form
\begin{gather*}
  f_p
\;:=\;
  \sum_{\substack{n \in \ZZ\\\gcd(n,M_p^\#) = \gcd(\beta,M_p^\#)}}
  c(f;\,n) e(n\tau)
\tx{,}
\end{gather*}
where factor~$M = M_p M_p^\#$ with a~$p$\nbd power~$M_p$ and~$M_p^\#$ co-prime to~$p$.

\item
In the situation of Theorem~\ref{mainthm:ramanujan-type-implies-hecke}, the conditions on~$\lambda_p$ formulated in Proposition~\ref{prop:hecke-implies-ramanujan-type} will be satisfied. See also the remark following Theorem~\ref{thm:ramanujan-type-implies-hecke}.

\item
If~$\rmTcl_p$ acts semi-simply modulo~$\ell$ on~$\rmS_k(\chi, \OKell)$, where we use common notation revisited in Section~\ref{sec:abstract-spaces-of-modular-forms}, then Proposition~\ref{prop:hecke-implies-ramanujan-type} informs us that all~$f \in \rmS_k(\chi, \OKell)$ exhibit Ramanujan-type congruences modulo~$\ell$ on~$p^m (p \ZZ + \beta_0)$ for some positive integer~$m$ and all~$\beta_0 \in \ZZ \setminus p \ZZ$. The choice
\begin{gather*}
  m = \ell^{2 \dim\,\rmS_k(\chi, \OKell)} - 2
\end{gather*}
always works, but depending on the inertia at~$\ell$ of the coefficient field of Hecke eigenforms in~$\rmS_k(\chi)$, smaller~$m$ might be suitable.

\item
The theorem can be extended to the case of Eisenstein series, but would then require a slightly more technical statement due to the possibility that Eisenstein series are congruent to constants. For example, we have~$E_4 \equiv 1 \,\pmod{5}$.

\item
A weaker and significantly more technical statement that subsumes the case of weakly holomorphic modular forms in given in Proposition~\ref{prop:ramanujan-type-implies-hecke-like}.
\end{mainremarkenumerate}

For a prime~$\ell$, positive integers~$M$ and~$m$, and a weakly holomorphic modular form~$f$ with~$\ell$-integral Fourier coefficients, one can consider the set~$P$ of primes~$p$ with the property that
\begin{gather*}
  \exists \beta \in \ZZ \,:\,
  \forall n \in \ZZ \,:\,
  c(f;\, M p^m n + \beta)
\equiv
  0\;\pmod{\ell}
\tx{.}
\end{gather*}
The aforementioned result by Ono~\cite{ono-2000} can be rephrased to the effect that the upper density of~$P$ with $f = 1 \slash \eta(24 \tau)$, $M = \ell$, and $m = 4$ is positive, if~$\ell \ge 5$. In reaction to an earlier version of the present work, Serre asked whether~$P$ has a density. In the situation of Theorem~\ref{mainthm:hecke-stable} this is indeed the case.
\begin{maincorollary}
\label{maincor:ramanujan-type-has-density}
Fix an odd prime~$\ell$, and let~$f$ be a cusp form of integral weight~$k$ for a Dirichlet character~$\chi$ modulo~$N$ with~$\ell$\nbd integral Fourier coefficients~$c(f;\,n)$. Fix positive integers~$M$ and~$m$ and an integer~$\beta$. Then the set of primes~$p$ such that
\begin{gather*}
  \exists \beta \in \ZZ \,:\,
  \forall n \in \ZZ \,:\,
  c(f;\,M p^m n + \beta) \equiv 0 \;\pmod{\ell}
\end{gather*}
is frobenian in the sense of~\S\,3.3.2 of Serre~\cite{serre-2012}. In particular, it has a density.
\end{maincorollary}
\begin{mainremark}
A statement subsuming the case of algebraic Fourier coefficients is given in Corollary~\ref{cor:ramanujan-type-has-density}.
\end{mainremark}

The structure of congruences for Hecke eigenvalues as in Theorem~\ref{mainthm:ramanujan-type-implies-hecke} as opposed to Ramanujan-type congruences is much more restricted. It is known by work of Radu~\cite{radu-2012,radu-2013} that Ramanujan-type congruences for~$1 \slash \eta$ naturally live on square-classes of~$\beta \,\pmod{M}$. At this level of generality Radu's result is optimal, but congruences in integral weights adhere to much tighter rules. We offer three structure results, which also form the basis of our proofs of Theorems~\ref{mainthm:hecke-stable} and~\ref{mainthm:ramanujan-type-implies-hecke}.
\begin{maintheorem}
\label{mainthm:structure-of-ramanujan-type-congruences}
Fix an odd prime~$\ell$ and let~$f$ be a weakly holomorphic modular form of integral weight~$k$ for a Dirichlet character~$\chi$ modulo~$N$ with~$\ell$\nbd integral Fourier coefficients~$c(f;\,n)$. Assume that~$f$ satisfies the Ramanujan-type congruence
\begin{gather*}
  \forall n \in \ZZ \,:\, c(f;\,M n + \beta) \equiv 0 \;\pmod{\ell}
\end{gather*}
for some positive integer~$M$ and some integer~$\beta$. Consider a prime~$p \isdiv M$ that does not divide~$\ell N$, and factor~$M$ as~$M_p M_p^\#$ with a~$p$\nbd power~$M_p$ and~$M_p^\#$ co-prime to~$p$.

\begin{enumeratearabic}
\item
\label{it:mainthm:structure-of-ramanujan-type-congruences:gaps}
We have the Ra\-ma\-nu\-jan-type congruence with gap
\begin{gather*}
  \forall n \in \ZZ \setminus p \ZZ \,:\,
  c(f;\,(M \slash p) n + M_p \beta') \equiv 0 \;\pmod{\ell}
\quad
\tx{with }
  M_p \beta' \equiv \beta \;\pmod{M_p^\#}
\tx{.}
\end{gather*}

\item	
\label{it:mainthm:structure-of-ramanujan-type-congruences:square-free}
We have the Ramanujan-type congruence
\begin{gather*}
  \forall n \in \ZZ \,:\,
  c(f;\,M' n + \beta) \equiv 0 \;\pmod{\ell}
\quad
\tx{with }
  M' = \gcd(M, M_{\rms\rmf} \beta )
\tx{,}
\end{gather*}
where~$M_{\rms\rmf}$ is the largest square-free divisor of\/~$M$.

\item
\label{it:mainthm:structure-of-ramanujan-type-congruences:remove-prime}
If\/~$M_p \isdiv \beta$, then we have the Ramanujan-type congruence
\begin{gather*}
  \forall n \in \ZZ \,:\, c(f;\,M_p^\# n + \beta) \equiv 0 \;\pmod{\ell}
\tx{.}
\end{gather*}
If\/~$p^2 \nisdiv M$, then we have the congruence
\begin{gather*}
  \forall n \in \ZZ, M_p^\# n + \beta \ne 0 \,:\,
  c(f;\,M_p^\# n + \beta) \equiv 0 \;\pmod{\ell}
\tx{,}
\end{gather*}
which is a Ramanujan-type congruence if~$M_p^\# \nisdiv \beta$.
\end{enumeratearabic}
\end{maintheorem}
\begin{mainremarkenumerate}
\item
A statement subsuming the case of algebraic Fourier coefficients is given in Theorem~\ref{thm:structure-of-ramanujan-type-congruences}.

\item
The congruence in~\ref{it:mainthm:structure-of-ramanujan-type-congruences:gaps} of Theorem~\ref{mainthm:structure-of-ramanujan-type-congruences} implies that Ramanujan-type congruences exist for full square-classes of~$\beta \,\pmod{M}$.

\item
The congruence in~\ref{it:mainthm:structure-of-ramanujan-type-congruences:square-free} of Theorem~\ref{mainthm:structure-of-ramanujan-type-congruences} is stronger at the prime~$2$ than the one implied directly by the fact that Ramanujan-type congruences exist for full square-classes of~$\beta \,\pmod{M}$; see Proposition~\ref{prop:congruences-square-free-quotient}.

\item
The first congruence in~\ref{it:mainthm:structure-of-ramanujan-type-congruences:remove-prime} of Theorem~\ref{mainthm:structure-of-ramanujan-type-congruences} allows us to remove~$\rmU_p$\nbd congruences in the case of~$p \nisdiv \ell N$.

\item
The second congruence in~\ref{it:mainthm:structure-of-ramanujan-type-congruences:remove-prime} of Theorem~\ref{mainthm:structure-of-ramanujan-type-congruences} allows us to remove the square-free part from~$M$ that is co-prime to~$\ell N$. This answers the analogue in integral weights of a question examined by Ahlgren, Beckwith, and the author for the partition function~\cite{ahlgren-beckwith-raum-2020-preprint}.

\item
Combining the congruence in~\ref{it:mainthm:structure-of-ramanujan-type-congruences:square-free} and the second congruence in~\ref{it:mainthm:structure-of-ramanujan-type-congruences:remove-prime} of Theorem~\ref{mainthm:structure-of-ramanujan-type-congruences}, we can remove primes from~$M$ that do not divide~$\ell N \beta$.

\item
The condition $M^\#_p n + \beta \ne 0$ that appears in the second congruence in~\ref{it:mainthm:structure-of-ramanujan-type-congruences:remove-prime} of Theorem~\ref{mainthm:structure-of-ramanujan-type-congruences} is related to congruences between Eisenstein series and constants. If~$\beta$ is not divisible by~$M_p^\#$ it is vacuous.
\end{mainremarkenumerate}

It is hard in general to obstruct Ramanujan-type congruences. While there are techniques known to provide them, very few results exist that rule them out. As an example of how our characterization of Ramanujan-type congruences can be employed, we showcase the next corollary and example, whose proofs are short once Theorem~\ref{mainthm:ramanujan-type-implies-hecke} is established and are given at the end of Section~\ref{sec:ramanujan-type-hecke:hecke-congruence}.
\begin{maincorollary}
\label{maincor:no-ramanujan-type-congruences}
Fix an odd prime~$\ell$ and let~$f$ be a modular form of integral weight~$k$ for a Dirichlet character~$\chi$ modulo~$N$ with~$\ell$\nbd integral Fourier coefficients~$c(f;\,n)$. For an integer~$m \ge 2$ with~$\gcd(\ell (\ell - 1), m+1) = 1$ and a prime~$p \nisdiv \ell N$ such that~$\chi(p) p^{k-1}$ is not a square modulo~$\ell$, the Ramanujan-type congruence 
\begin{gather*}
  \forall n \in \ZZ \setminus p \ZZ \,:\,
  c(f;\,p^m n) \equiv 0
  \;\pmod{\ell}
\end{gather*}
implies that~$f \equiv 0 \,\pmod{\ell}$.
\end{maincorollary}
\begin{mainremark}
Observe that by Theorem~\ref{mainthm:structure-of-ramanujan-type-congruences} it suffices to consider Ramanujan-type congruences with gap as in Corollary~\ref{maincor:no-ramanujan-type-congruences} to cover all Ramanujan-type congruences on arithmetic progressions~$p^{m+1} \ZZ + \beta$, $m$ a nonnegative integer.
\end{mainremark}

\begin{mainexample}
\label{mainex:delta-function}
Theorems~\ref{mainthm:ramanujan-type-implies-hecke} and~\ref{mainthm:structure-of-ramanujan-type-congruences} are particularly strong in the case of level~$1$, when they explain all congruences apart from $\rmU_\ell$-congruences. For instance, consider an odd prime~$\ell$, the Ramanujan~$\Delta$-function, and the set of all arithmetic progressions~$M \ZZ + \beta$ with the property that~$c(\Delta; M n + \beta) \equiv 0 \,\pmod{\ell}$ for all integers~$n$. Then the elements of this set that are maximal with respect to the ordering by inclusion are of the form
\begin{gather*}
  \ell \ZZ + \beta
\tx{,}\; \beta \in \ZZ
\tx{,}
\quad\tx{or}\quad
  p^m (p \ZZ + \beta)
\tx{,}\;
  p \ne \ell \tx{ prime},\,
  m, \beta \in \ZZ,\,
  m \ge 1,\, p \nisdiv \beta
\tx{.}
\end{gather*}
In the next table we list some of the maximal elements for various~$\ell$ and~$p$. For simplicity, we write $p^m \ZZ_p^\times$ for the collection of all arithmetic progressions~$p^m (p \ZZ + \beta)$, where $\beta$ is any integer not divisible by~$p$.
\begin{center}
\begin{tabular}{r@{}r@{\hspace{1em}}rrrrr}
\toprule
&& $p = 2$ & $p = 3$ & $p = 5$ & $p = 7$ & $p = 11$ \\
\midrule
$\ell ={}$ & $ 3$ & $2^1 \ZZ^\times_2$, $2^3 \ZZ^\times_2$    & $3 \ZZ$                             &
                    $5^1 \ZZ^\times_5$, $5^3 \ZZ^\times_5$ & $7^2 \ZZ^\times_7$, $7^5 \ZZ^\times_7$ & $11^1 \ZZ^\times_{11}$, $11^3 \ZZ^\times_{11}$ \\
$\ell ={}$ & $ 5$ & $2^3 \ZZ^\times_2$, $2^7 \ZZ^\times_2$ & $3^3 \ZZ^\times_3$, $3^7 \ZZ^\times_3$ &
                    $5 \ZZ$                                & $7^3 \ZZ^\times_7$, $7^7 \ZZ^\times_7$ & $11^4 \ZZ^\times_{11}$, $11^9 \ZZ^\times_{11}$ \\
$\ell ={}$ & $ 7$ & $2^6 \ZZ^\times_2$, $2^{13} \ZZ^\times_2$ & $3^1 \ZZ^\times_3$, $3^3 \ZZ^\times_3$ &
                    $5^1 \ZZ^\times_5$, $5^3 \ZZ^\times_5$    & $7 \ZZ$                                & $11^6 \ZZ^\times_{11}$, $11^{13} \ZZ^\times_{11}$ \\
$\ell ={}$ & $11$ & $2^3 \ZZ^\times_2$, $2^7 \ZZ^\times_2$    & $3^{10} \ZZ^\times_3$, $3^{21} \ZZ^\times_3$ &
                    $5^4 \ZZ^\times_5$, $5^9 \ZZ^\times_5$    & $7^9 \ZZ^\times_7$, $7^{19} \ZZ^\times_7$    & \\
\bottomrule
\end{tabular}
\end{center}
\end{mainexample}

Theorems~\ref{mainthm:ramanujan-type-implies-hecke} and~\ref{mainthm:structure-of-ramanujan-type-congruences} still leave ample room for Ramanujan-type congruences in integral weight. While Theorem~\ref{mainthm:ramanujan-type-implies-hecke} provides a strong relation to Hecke operators that we can utilize in customary ways, the condition that~$p$ does not divide~$\ell N$ in both theorems provides some flexibility in constructing examples where the~$\gcd$ of~$M$ and the level is non-trivial. Cases with~$\ell \isdiv M$ arise from a construction through~$\Theta$\nbd operators that was provided by Dewar~\cite{dewar-2012}.  Moreover, Treneer provided an infinite family of Ramanujan-type congruences~\cite{treneer-2006} in the spirit of Ahlgren-Ono using Hecke operators. We revisit and amend these prior constructions of Ramanujan-type congruences in Sections~\ref{sec:ramanujan-type-hecke:theta-operator} and~\ref{sec:ramanujan-type-hecke:implied-by-hecke}. As a result, we find that essentially all congruences allowed by Theorems~\ref{mainthm:ramanujan-type-implies-hecke} and~\ref{mainthm:structure-of-ramanujan-type-congruences} occur for some modular form.

The proofs of Theorems~\ref{mainthm:hecke-stable},~\ref{mainthm:ramanujan-type-implies-hecke}, and~\ref{mainthm:structure-of-ramanujan-type-congruences}, and of Corollary~\ref{maincor:ramanujan-type-has-density}, which are special cases of Theorems~\ref{thm:hecke-module-ramanujan-type-congruences},~\ref{thm:ramanujan-type-implies-hecke}, and~\ref{thm:structure-of-ramanujan-type-congruences} and Corollary~\ref{cor:ramanujan-type-has-density}, build upon the concept of abstract spaces of modular forms, which we introduce in this paper. These spaces can be thought of as the dual of vector-valued modular forms. They allow for the intervention of modular representation theory into the subject. While we focus on the case of integral weights in this paper, Ramanujan-type congruences can be investigated through this formalism also in half-integral weights. Section~\ref{sec:characterization-ramanujan-type-congruences} contains a series of examples, which highlight the representation theoretic features of congruences obtained by Atkin and by Ono. In particular, we find in Example~\ref{ex:ovell-kernel-half-integral-weight:ono} that some Ramanujan-type congruences in half-integral weights necessarily exist on the union of two-square classes of~$\beta \,\pmod{M}$ as opposed to a single one. This parallels our results for integral weights in Theorem~\ref{mainthm:structure-of-ramanujan-type-congruences}, but does not hold in general as Atkin's congruences in Example~\ref{ex:ovell-kernel-half-integral-weight:atkin} illustrate. The case of half-integral weights is different from the case of integral weights in a subtle way. It will be the subject of a sequel to this work.

\paragraph{Acknowledgment}

The author thanks Claudia Alfes-Neumann, Scott Ahlgren, Olivia Beckwith, Henri Darmon, Olav Richter, and Jean-Pierre Serre for inspiring discussions and helpful comments.

\section{Abstract spaces of modular forms}
\label{sec:abstract-spaces-of-modular-forms}

Throughout this paper, we write $e(z)$ for $\exp(2 \pi i\, z)$ for $z \in \CC$ or~$z \in \CC \slash \ZZ$. 

\paragraph{The upper half plane}

The Poincar\'e upper half plane is defined as
\begin{gather*}
  \HS
\;:=\;
  \big\{ \tau \in \CC \,:\, \Im(\tau) > 0 \big\}
\tx{.}
\end{gather*}
It carries an action of~$\GL{2}^+(\RR) := \{ \ga \in \GL{2}(\RR) \,:\, \det(\ga) > 0 \}$ by M\"obius transformations
\begin{gather*}
  \begin{pmatrix} a & b \\ c & d \end{pmatrix} \tau
\;:=\;
  \frac{a \tau + b}{c \tau + d}
\tx{.}
\end{gather*}

\paragraph{The metaplectic group}

While the main theorems in this paper can be stated and proved without reference to the metaplectic group, Section~\ref{sec:characterization-ramanujan-type-congruences} extends from~$\SL{2}(\ZZ)$ to~$\Mp{1}(\ZZ)$ without further complication. We include the metaplectic case, since this level of generality will be needed in the second part of this series of papers.

The general metaplectic group~$\GMp{1}(\RR)$ is a nontrivial central extension
\begin{gather*}
  1 \lra \mu_2 \lra \GMp{1}(\RR) \lra \GL{2}(\RR) \lra 1
\tx{,}\quad
  \mu_{2} := \{ \pm 1 \}
\tx{.}
\end{gather*}
A common realization of~$\GMp{1}(\RR)$ is given by
\begin{gather*}
  \GMp{1}(\RR)
\,:=\,
  \big\{
  (\ga, \omega) \,:\,
  \ga = \begin{psmatrix} a & b \\ c & d \end{psmatrix} \in \GL{2}(\RR),\,
  \omega :\, \HS \ra \CC,\,
  \omega(\tau)^2 = c \tau + d
  \big\}
\tx{,}
\end{gather*}
in which case multiplication in~$\GMp{1}(\RR)$ is given by
\begin{gather*}
  (\ga, \omega) \cdot (\ga', \omega')
\,:=\,
  (\ga \ga',\, \omega \circ \ga' \cdot \omega')
\tx{.}
\end{gather*}
Throughout, the letter~$\ga$ may denote elements of both~$\GL{2}(\RR)$ and~$\GMp{1}(\RR)$.

We let~$\GMp{1}(\QQ)$ and~$\Mp{1}(\ZZ)$ be the preimages of~$\GL{2}(\QQ)$ and~$\SL{2}(\ZZ)$ under the projection~$\GMp{1}(\RR) \ra \GL{2}(\RR)$. In particular, we consider $\GMp{1}(\QQ)$ as a subgroup of~$\GMp{1}(\RR)$ as opposed to the general metaplectic group over the (finite) adeles, $\GMp{1}(\bbA_\rmf)$ or $\GMp{1}(\bbA)$.

\paragraph{Congruence subgroups}

A subgroup~$\Ga \subseteq \Mp{1}(\ZZ)$ is called a congruence subgroup if its projection to~$\SL{2}(\ZZ)$ is a congruence subgroup. We fix notation for the special cases
\begin{align*}
  \Ga_0(N)
\;&:=\;
  \big\{
  \begin{psmatrix} a & b \\ c & d \end{psmatrix} \in \SL{2}(\ZZ)
  \,:\,
  c \equiv 0 \,\pmod{N}
  \big\}
\tx{,}
\\
  \Ga(N)
\;&:=\;
  \big\{
  \begin{psmatrix} a & b \\ c & d \end{psmatrix} \in \SL{2}(\ZZ)
  \,:\,
  b,c \equiv 0 \,\pmod{N},\, a,d \equiv 1 \,\pmod{N}
  \big\}
\tx{,}
\\
  \Ga^+_\infty
\;&:=\;
  \big\{
  \begin{psmatrix} 1 & b \\ 0 & 1 \end{psmatrix} \in \SL{2}(\ZZ)
  \big\}
\tx{,}
\end{align*}
and write~$\tdGa_0(N)$, $\tdGa(N)$, and~$\tdGa^+_\infty$ for their preimages under the projection from~$\Mp{1}(\ZZ)$ to~$\SL{2}(\ZZ)$.

We record for later use that the integral metaplectic group~$\Mp{1}(\ZZ)$ is generated by the elements
\begin{gather*}
  S
\,:=\,
  \big( \begin{psmatrix} 0 & -1 \\ 1 & 0 \end{psmatrix}, \tau \mto \sqrt{\tau} \big)
\tx{,}\quad
  T
\,:=\,
  \big( \begin{psmatrix} 1 & 1 \\ 0 & 1 \end{psmatrix}, \tau \mto 1 \big)
\tx{,}
\end{gather*}
where $\sqrt{\tau}$ is the principal branch of the holomorphic square root on~$\HS$. We allow ourselves to also write~$T$ for its image in~$\SL{2}(\ZZ)$.

\paragraph{Slash actions}

The action~$\GL{2}^+(\RR) \circlearrowright \HS$ in conjunction with the cocycle $(\ga,\om) \mto \om$ yields the slash actions of weight~$k \in \frac{1}{2}\ZZ$ on functions~$f :\, \HS \ra \CC$:
\begin{gather}
\label{eq:slash-action-mp1r}
  f \big|_k\, (\ga,\om)
\,:=\,
  \det(\ga)^{\frac{k}{2}} \omega^{-2k} \cdot f \circ \ga
\tx{,}\quad
  (\ga,\om) 
\in
  \GMp{1}^+(\RR)
\tx{.}
\end{gather}

\paragraph{Modular forms}

Fix $k \in \frac{1}{2}\ZZ$, a finite index subgroup $\Ga \subseteq \Mp{1}(\ZZ)$, and a finite order character~$\chi : \Ga \ra \CC^\times$. A weakly holomorphic modular form of weight~$k$ for the character~$\chi$ on~$\Ga$ is a holomorphic function~$f :\, \HS \ra \CC$ satisfying the following two conditions:
\begin{enumerateroman}
\item For all $\ga \in \Ga$, we have $f \big|_k\, \ga = \chi(\ga) f$.
\item There exists $a \in \RR$ such that for all~$\ga \in \Mp{1}(\ZZ)$ we have
\begin{gather*}
  \big| \big( f \big|_k\, \ga \big) (\tau) \big| < b \exp\big( -a \Im(\tau) \big)
  \tx{\ as\ }\tau \ra i \infty
\tx{.}
\end{gather*}
\end{enumerateroman}
We write~$\rmM^!_k(\Ga, \chi)$ for the space of weakly holomorphic modular forms of weight~$k$ for the character~$\chi$ on~$\Ga$. We write~$\rmM_k(\Ga,\chi)$ and~$\rmS_k(\Ga,\chi)$ for the corresponding spaces of modular forms and cusp forms. If~$\chi$ is trivial, we suppress it from our notation. 

When saying that~$f$ is a modular form for a Dirichlet character~$\chi$ modulo~$N$, we will always assume that the level of~$f$ divides~$N$.

Let~$\Ga \subseteq \Mp{1}(\ZZ)$ be a normal subgroup of finite index. The spaces of weakly holomorphic modular forms for~$\Ga$ naturally carry the structure of a right-representation for~$\Mp{1}(\ZZ)$ via the weight-$k$ slash action. If~$k$ is integral, this representations factors through~$\Mp{1}(\ZZ) \thra \SL{2}(\ZZ)$.

\paragraph{Fourier expansions}

Given a finite index subgroup~$\Ga \subseteq \Mp{1}(\ZZ)$ and a finite order character~$\chi$ of~$\Ga$, there is a positive integer~$N$ such that $T^N \in \ker(\chi)$. In particular, a weakly holomorphic modular form~$f$ for this character~$\chi$ has a Fourier series expansion of the form
\begin{gather}
\label{eq:fourier-expansion}
  f(\tau)
\;=\;
  \sum_{n \in \frac{1}{N} \ZZ}
  c(f;\,n) e(n \tau)
\tx{.}
\end{gather}
For convenience, we set~$c(f;\,n) = 0$ if~$n \in \QQ \setminus \frac{1}{N}\ZZ$.

Given any ring~$R$, we write
\begin{gather}
  \rmFE(R)
\;:=\;
  R\big\llbrkt q^{\frac{1}{\infty}} \big\rrbrkt \big[q^{-1}\big]
\end{gather}
for the ring of Puiseux series with coefficients in~$R$. The Fourier expansion of weakly holomorphic modular forms yields a map
\begin{gather}
\label{eq:def:fourier-expansion-map}
  \rmfe :\,
  \rmM^!_k(\Ga, \chi)
\lra
  \rmFE(\CC)
\tx{,}\quad
  f
\lmto
  \rmfe(f)
:=
  \sum_{n \in \QQ} c(f;\,n) q^n
\tx{.}
\end{gather}

We record that if~$T \in \Ga$, then this map is $\tdGa^+_\infty$\nbd equivariant for the slash action~$|_k$ on the left hand side and the action defined by
\begin{gather*}
  q^n \big| T
=
  \chi(T) e(n) q^n
\tx{.}
\end{gather*}

We let~$\OKell$ be the localization of the integers~$\OK$ in a number field~$K \subset \CC$ at a given ideal~$\ell \subset \OK$. If $K = \QQ$ and therefore~$\OK = \ZZ$, we write $\ZZ_\ell = \OKell$. We set
\begin{align*}
  \rmM^!_k(\Ga, \chi;\, \OKell)
\;&:=\;
  \big\{
  f \in \rmM^!_k(\Ga, \chi) \,:\,
  \rmfe(f) \in \rmFE(\OKell)
  \big\}
\tx{,}
\\
  \rmM_k(\Ga, \chi;\, \OKell)
\;&:=\;
  \rmM_k(\Ga, \chi)
  \cap
  \rmM^!_k(\Ga, \chi;\, \OKell)
\tx{,}
\\
  \rmS_k(\Ga, \chi;\, \OKell)
\;&:=\;
  \rmS_k(\Ga, \chi)
  \cap
  \rmM^!_k(\Ga, \chi;\, \OKell)
\tx{.}
\end{align*}

\paragraph{The operators \texpdf{$\rmU_p$}{Up} and \texpdf{$\rmV_p$}{Vp}}

Given a prime~$p$ and a modular form~$f \in \rmM_k(\Ga_0(N), \chi)$ for a Dirichlet character~$\chi$, we set
\begin{gather}
  \big( f \big| \rmU_p \big)(\tau)
\;:=\;
  \sum_n c(f; p n) e(n \tau)
\quad\tx{and}\quad \big( f \big| \rmV_p \big)(\tau)
\;:=\;
  \sum_n c(f; n) e(p n \tau)
\tx{.}
\end{gather}
A calculation shows that~$f = f | \rmV_p \rmU_p$ and~$f | \rmU_p \rmV_p \in \rmM_k(\Ga_0(p^2 N), \chi)$. We can also verify by a calculation that~$\rmV_p \rmU_p$ commutes with~$\rmTcl_q$ if~$p \ne q$.

\paragraph{Group algebras and induction}

Given a commutative ring~$R$, we write $R[\Ga]$ for the group algebra of a discrete group~$\Ga$, let $\fraku_\ga \in R[\Ga]$ for $\ga \in \Ga$ denote its canonical units, and thus identify $\Ga$ with a subgroup of~$R[\Ga]^\times$. Right-modules of $R[\Ga]$ are in canonical correspondence to right-representations of~$\Ga$ over~$R$. Given a right-module~$V$ of~$R[\Ga]$, we write $\ker_{\Ga}(V)$ for the subgroup of~$\Ga$ that acts trivially on~$V$.

Fix a finite index subgroup~$\Ga' \subseteq \Ga$. Since all representations in this work factor through finite groups, we allow ourselves to define the induction of a right-module~$V$ for $R[\Ga']$ as
\begin{gather*}
  \Ind_{\Ga'}^\Ga\,V
:=
  V \otimes_{R[\Ga']} R[\Ga]
\tx{,}
\end{gather*}
which is isomorphic to the common definition in terms of $\Ga'$\nbd covariant functions on~$\Ga$. We omit sub- and superscripts when we believe that they are clear from the context.

\subsection{Definition of abstract spaces of modular forms}

Given a ring~$R \subseteq \CC$ and a finite index subgroup~$\Ga \subseteq \Mp{1}(\ZZ)$, let $V$ be a finite dimensional right-module for $R[\Ga]$ such that $\ker_{\Ga}(V) \subseteq \Mp{1}(\ZZ)$ has finite index. An abstract module (or space, if $R$ is a field) of weakly holomorphic modular forms over~$R$ is a pair of such a module~$V$ and a homomorphism of~$R[\Ga]$\nbd modules $\phi : V \ra \rmM^!_k(\ker_{\Ga}(V))$ for some~$k \in \frac{1}{2} \ZZ$. We say that the abstract module of weakly holomorphic modular forms~$(V,\phi)$ is realized in weight~$k$. If the image of~$\phi$ consists of modular forms or cusp forms, we refer to~$(V,\phi)$ as an abstract module of modular forms or cusp forms. If not stated differently we will assume that~$\Ga = \Mp{1}(\ZZ)$. Typically, we will have~$R = \CC$ or $R = \OKell$.

Given an abstract module of weakly holomorphic modular forms~$V$ for~$\Ga$, we write the action of~$\ga \in \Ga$ on~$v \in V$ as~$v | \ga$, in order to emphasis the relation to the slash action on modular forms.

\paragraph{Induction of modular forms}

Fix a weakly holomorphic modular form~$f \in \rmM^!_k(\Ga, \chi)$ for a character~$\chi$ of~$\Ga \subseteq \Mp{1}(\ZZ)$. The inclusion map
\begin{gather*}
  \phi :\,
  R f \lra \rmM^!_k\big( \ker(\chi) \big)
\tx{,}\;
  f \lmto f
\end{gather*}
yields an abstract module of weakly holomorphic modular forms~$(R f, \phi)$ for every ring~$R \subseteq \CC$ with $\ZZ[\chi] \subseteq R$, where $\ZZ[\chi]$ arises from~$\ZZ$ by adjoining the image of~$\chi$.

We also obtain an abstract module of weakly holomorphic modular forms
\begin{gather}
\label{eq:def:induction}
  \Ind^{\Mp{1}(\ZZ)}_\Ga\,R f
\;:=\;
  \Ind^{\Mp{1}(\ZZ)}_\Ga\,\big(R f, \phi\big)
\;:=\;
  \Big(
  R f \otimes_{R[\Ga]} R\big[\Mp{1}(\ZZ) \big],\,
  f \otimes \fraku_\ga \mto f \big|_k\,\ga
  \Big)
\tx{.}
\end{gather}
As for the induction of representations, we allow ourselves to omit sub- and superscripts.

\begin{example}
Let~$k$ and~$N > 0$ be integers, and consider a weakly holomorphic modular form~$f \in \rmM_k(\Ga_0(N),\chi)$ for a Dirichlet character~$\chi$ modulo~$N$, which yields a character of~$\Ga_0(N)$ via~$\begin{psmatrix} a & b \\ c & d \end{psmatrix} \mto \chi(d)$, as usual. We view $f$ as a modular form for~$\tdGa_0(N)$ via its projection to~$\Ga_0(N)$. Consider~$(V,\phi) = \Ind_{\tdGa_0(N)}\,\CC f$.  The co-sets $\ga \in \Ga_0(N) \backslash \SL{2}(\ZZ)$ or, more specifically, $(\ga,\omega) \in \tdGa_0(N) \backslash \Mp{1}(\ZZ)$ yield a basis~$f \otimes \frake_{(\ga,\omega)}$ of~$V$. The realization map~$\phi$ can be viewed as encoding the cusp expansions of~$f$. Observe that $V$ might contain vectors that map to zero under~$\phi$. For instance, if $f$ is a newform and~$N \ne 1$, then
\begin{gather*}
  \phi\Big(
  f \otimes
  \big(
  \sum_{(\ga,\omega) \in \tdGa_0(N) \backslash \Mp{1}(\ZZ)}
  \frake_{(\ga,\omega)}
  \big)
  \Big)
=
  \sum_{(\ga,\omega) \in \tdGa_0(N) \backslash \Mp{1}(\ZZ)}
  f \big|_k \ga
=
  0
\tx{.}
\end{gather*}
\end{example}

\subsection{Hecke operators}
\label{ssec:hecke-operators}

While we allowed ourselves to denote the classical Hecke operator by~$\rmT_p$ in the introduction, in the remainder of this paper we will write~$\rmTcl_p$ for it, and let~$\rmT_p$ be the vector-valued Hecke operator defined in this section.

Given a positive integer~$M$, we let
\begin{gather*}
  \GL{2}^{(M)}(\ZZ)
:=
  \big\{
  \ga \in \Mat{2}(\ZZ) \,:\,
  \det(\ga) = M
  \big\}
\tx{,}
\end{gather*}
and correspondingly write $\GMp{1}^{(M)}(\ZZ)$ for its preimage in~$\GMp{1}(\RR)$ under the projection to~$\GL{2}(\RR)$. For a commutative ring~$R$, we let $R[ \GMp{1}^{(M)}(\ZZ) ]$ be the~$R[\Mp{1}(\ZZ)]$\nbd bi-module with $R$\nbd basis~$\fraku_\ga$, $\ga \in \GMp{1}^{(M)}(\ZZ)$. Given an abstract module of weakly holomorphic modular forms~$(V,\phi)$ over~$R \subseteq \CC$, we define the $M$\thdash\ Hecke operator by
\begin{gather}
\label{eq:def:hecke-operator-abstract-space-of-modular-forms}
  \rmT_M\,(V,\phi)
\;:=\;
  \Big(
  V \otimes_{R[\Mp{1}(\ZZ)]} R\big[ \GMp{1}^{(M)}(\ZZ) \big],\;
  v \otimes \fraku_\ga \mto \phi(v) \big|_k\,\ga
  \Big)
\tx{.}
\end{gather}
It is straightforward to verify that~$\rmT_M\,(V,\phi)$ is an abstract module of weakly holomorphic modular forms. By slight abuse of notation, we write $\rmT_M\,(V,\phi) = (\rmT_M\,V,\rmT_M\,\phi)$, and further write~$\phi$ instead of~$\rmT_M\,\phi$, if no confusion can arise.

We can extend the notion of vector-valued Hecke operators to abstract modules of weakly holomorphic modular forms associated with congruence subgroups~$\tdGa_0(N)$. Specifically, given co-prime positive integers~$M$ and~$N$, we replace~$\GL{2}^{(M)}(\ZZ)$ in the above construction by
\begin{alignat*}{2}
  \big\{
  \ga = \begin{psmatrix} a & b \\ c & d \end{psmatrix} \in \GL{2}^{(M)}(\ZZ)
  &\,{}:{}\,&
  c  \equiv 0 &\,\pmod{N}
  \big\}
\tx{.}
\end{alignat*}
The isomorphisms in~\eqref{eq:hecke-operator-decomposition} and~\eqref{eq:hecke-to-induced} below extend to this settings.

\paragraph{Decomposition of Hecke operators}

Given co-prime positive integers~$M_1$ and~$M_2$ and an abstract module of modular forms~$(V,\phi)$, we have an isomorphism
\begin{gather}
\label{eq:hecke-operator-decomposition}
  \rmT_{M_1}\,\rmT_{M_2} (V,\phi)
\lra
  \rmT_{M_1 M_2} (V,\phi)
\tx{,}\quad
  \big( v \otimes \fraku_{\ga_2} \big) \otimes \fraku_{\ga_1}
\lmto
  v \otimes \fraku_{\ga_2 \ga_1}
\tx{,}
\end{gather}
which arises from the isomorphism of~$\Mp{1}(\ZZ)$\nbd bi-sets
\begin{gather*}
  \GMp{1}^{(M_2)}(\ZZ)
  \,\times_{\Mp{1}(\ZZ)}\,
  \GMp{1}^{(M_1)}(\ZZ)
\lra
  \GMp{1}^{(M_1 M_2)}(\ZZ)
\tx{,}\quad
  (\ga_2, \ga_1)
\lmto
  \ga_2 \ga_1
\tx{.}
\end{gather*}

\paragraph{Connection to induced representations}

For a positive integer~$M$, consider the map
\begin{gather}
\label{eq:def:hecke-operators:psi-M}
  \psi_M :\,
  \tdGa_0(M) \lra \Mp{1}(\ZZ)
\tx{,}\;
  \ga
\lmto
  \begin{psmatrix} M & 0 \\ 0 & 1 \end{psmatrix}
  \ga
  \begin{psmatrix} 1 \slash M & 0 \\ 0 & 1 \end{psmatrix}
\tx{.}
\end{gather}
We write~$\psi_M^\ast\,V$ for the pullback along~$\psi_M$ of an~$\Mp{1}(\ZZ)$\nbd representation~$V$.

Consider the subset
\begin{gather*}
  \GL{2}^{(M)\times}(\ZZ)
\;:=\;
  \big\{
  \ga = \begin{psmatrix} a & b \\ c & d \end{psmatrix} \in \Mat{2}(\ZZ) \,:\,
  \det(\ga) = M,\,
  \gcd(a,b,c,d) = 1
  \big\}
\;\subseteq\;
  \GL{2}^{(M)}(\ZZ)
\end{gather*}
and the corresponding subset~$\GMp{1}^{(M)\times}(\ZZ)$ of~$\GMp{1}^{(M)}(\ZZ)$. Then we have a bijection of~$\Mp{1}(\ZZ)$\nbd right-sets
\begin{gather*}
  \tdGa_0(M) \big\backslash \Mp{1}(\ZZ)
\lra
  \Mp{1}(\ZZ) \big\backslash \GMp{1}^{(M)\times}(\ZZ)
\tx{,}\;
  \ga
\lmto
  \begin{psmatrix} M & 0 \\ 0 & 1 \end{psmatrix} \ga
\end{gather*}
that intertwines the associated cocycles via~$\psi_M$. As a consequence, for an abstract module of weakly holomorphic modular forms~$(V,\phi)$ for~$\Mp{1}(\ZZ)$, we have an inclusion 
\begin{gather}
\label{eq:hecke-to-induced}
\begin{aligned}
  \Ind\, \psi_M^\ast(V)
\;=\;
  V \otimes_{\CC[\tdGa_0(M)]} \CC\big[ \Mp{1}(\ZZ) \big]
&\lhra
  V \otimes_{\CC[\Mp{1}(\ZZ)]} \CC\big[ \GMp{1}^{(M)}(\ZZ) \big]
\;=\;
  \rmT_M\,V
\tx{,}
\\
  v \otimes \fraku_\ga
&\lmto
  v \otimes \fraku_{\begin{psmatrix} M & 0 \\ 0 & 1 \end{psmatrix}\ga}
\tx{.}
\end{aligned}
\end{gather}

More generally, given a positive integer~$N$ that is co-prime to~$M$ we have an analogous embedding~$\psi_M :\, \tdGa_0(M N) \lra  \tdGa_0(N)$. If~$V$ is a representation for~$\tdGa_0(N)$, we have
\begin{gather}
\label{eq:hecke-to-induced-subgroup}
  \Ind\, \psi_M^\ast(V)
\lhra
  \rmT_M\,V
\tx{,}
\;
  v \otimes \fraku_\ga
\lmto
  v \otimes \fraku_{\begin{psmatrix} M & 0 \\ 0 & 1 \end{psmatrix}\,\ga}
\tx{.}
\end{gather}

\section{Characterization of Ramanujan-type congruences}
\label{sec:characterization-ramanujan-type-congruences}

In Corollary~\ref{cor:congruence-on-arithmetic-progression-hecke-T-eigenvector} of this section, we derive a characterization of Ramanujan-type congruences for, say, a weakly holomorphic modular form~$f$ in terms of specific vectors in the abstract space of weakly holomorphic modular forms~$\rmT_M\,\CC f$. The key result of this section is Theorem~\ref{thm:ell-kernel-representation}, which later allows us to employ our characterization as illustrated in the discussion of ``nontrivial congruences'' in Section~\ref{ssec:nontrivial-congruences}.

\subsection{\texpdf{$\ell$}{l}-kernels}

Suppose that $(V,\phi)$ is an abstract module of weakly holomorphic modular forms. Then we obtain a Fourier expansion map~$\rmfe \circ \phi :\, V \lra \rmFE(\CC)$.

Given number field~$K \subset \CC$ with maximal order~$\OK$, recall that we write~$\OKell \subseteq K$ for the localization of~$\OK$ at an ideal~$\ell \subseteq \OK$. The $\ell$\nbd kernel of~$(V,\phi)$ is defined as the preimage
\begin{gather}
\label{eq:def:ell-kernel}
  \ellker\big( (V,\phi) \big)
\;:=\;
  \big( \rmfe \circ \phi \big)^{-1}\big( \ell\, \rmFE(\OKell) \big)
\tx{.}
\end{gather}
Observe that we suppress~$K$ from our notation, but it is implicitly given by~$\ell$. We allow ourselves to identify a rational integer~$\ell$ with the ideal in~$\ZZ$ that it generates.

The next example illustrates that we need a condition on the rationality of Fourier coefficients in order to employ the concept of~$\ell$\nbd kernels.
\begin{example}
The $\ell$\nbd kernel can be trivial: Consider the abstract module of modular forms~$\CC\,f$ associated with the modular form~$f = E_{12} + \pi \Delta$ of weight~$12$ and level~$1$.
\end{example}
We say that an abstract module of weakly holomorphic modular forms has Fourier expansions over~$K$, if the following condition holds:
\begin{gather}
\label{eq:def:fourier-expansions-over-K}
  V
\;=\;
  \big( \rmfe \circ \phi \big)^{-1} \big( \rmFE(K) \big) \otimes_K \CC
\tx{.}
\end{gather}
In other words, $V$ has a $K$\nbd structure that is compatible with Fourier expansions. In this situation, if~$V$ is a~$\Ga$ representation and~$\ker_\Ga(V)$ is a congruence subgroup, then the~$\ell$\nbd kernel spans~$V$ as a~$\CC$ vector space.

A priori, the $\ell$\nbd kernel of an abstract module of modular forms is merely an~$\OKell$-module. The key theorem of this paper is the next one, which guarantees that it is a module for a specific group algebra.
\begin{theorem}
\label{thm:ell-kernel-representation}
Let $K \subset \CC$ be a number field with fixed complex embedding, and let~$\ell$ be an ideal in~$\OK$. Let~$(V,\phi)$ be an abstract module of weakly holomorphic modular forms for~$\Ga \subseteq \Mp{1}(\ZZ)$ realized in weight~$k \in \frac{1}{2}\ZZ$. Assume that~$\ker_{\Ga}(V)$ is a congruence subgroup of level~$N$. Let $N_\ell$ be the smallest positive integer such that
\begin{gather}
\label{eq:thm:ell-kernel-representation:Nell}
  \gcd\big(\ell, N \slash N_\ell \big) = 1
  \tx{ if\/ $k \in \ZZ$}
\quad\tx{and}\quad
  \gcd\big(\ell, \lcm(4, N) \slash N_\ell \big) = 1
  \tx{ if\/ $k \in \tfrac{1}{2} + \ZZ$.}
\end{gather}
Then the $\ell$\nbd kernel of\/~$(V,\phi)$ is a right-module for~$\OKell[\tdGa_0(N_\ell) \cap \Ga]$.
\end{theorem}
\begin{remark}
Congruences of half-integral weight modular forms were also investigated by Jochnowitz in unpublished work~\cite{jochnowitz-2004-preprint}. The special case that~$N$ and~$\ell$ are co-prime can also be inferred from her results.
\end{remark}
\begin{proof}
Fix an element~$v \in \ellker((V,\phi))$ and write~$f = \phi(v)$. We have to show that $\phi(v | \ga) = f |_k\,\ga$ has Fourier coefficients in $\ell \OKell$ for every~$\ga \in \tdGa_0(N_\ell) \cap \Ga$. We let~$\Delta$ be the Ramanujan $\Delta$\nbd function in $\rmM_{12}(\Mp{1}(\ZZ))$, and observe that~$f \Delta |_k\,\ga$ has Fourier coefficients in~$\ell \OKell$ if and only if~$f |_k\,\ga$ does. In particular, we can assume that $f$ is a modular form after replacing $\phi$ with~$v \mto \phi(v) \Delta^h$ for sufficiently large~$h \in \ZZ$.

Consider the case $k \in \frac{1}{2} + \ZZ$. Let $\Theta$ be the $\CC$\nbd vector space spanned by the theta series
\begin{gather*}
  \theta_0(\tau)
\;:=\;
  \sum_{n \in \ZZ} e(n^2 \tau)
\tx{,}\quad
  \theta_1(\tau)
\;:=\;
  \sum_{n \in \frac{1}{2} + \ZZ} e(n^2 \tau)
\tx{.}
\end{gather*}
Observe that~$\Theta$ is a representation for~$\Mp{1}(\ZZ)$ under the slash action of weight~$\frac{1}{2}$, which is isomorphic to the dual of the Weil representation associated with the qua\-dra\-tic form~$n \mto n^2$. In particular, we can and will view~$\Theta$ as an abstract module of modular forms. Then $\ker_{\Mp{1}(\ZZ)}(\Theta)$ has level~$4$. We reduce ourselves to the case of integral weight~$k$ at the expense of replacing~$N$ by~$\lcm(4,N)$ when replacing~$(V,\phi)$ with the tensor product $(V,\phi) \otimes \Theta$. Indeed, $f|_k\, \ga$ has Fourier coefficients in~$\ell \OKell$ if and only if both $f \theta_0 |_{k+\frac{1}{2}}\,\ga$ and $f \theta_1 |_{k+\frac{1}{2}}\,\ga$ do.

In the remainder of the proof, we can and will assume that~$k \in \ZZ$. Since by our assumptions~$\ker_{\Ga}((V,\phi))$ has level~$N$, $f$ is a modular form for $\Ga(N) \subseteq \SL{2}(\ZZ)$. It therefore falls under Definition~VII.3.6 of modular forms in~\cite{deligne-rapoport-1973} if we suitably enlarge~$K$ by~$N$\thdash\ roots of unity. The transition between the two notions is outlined in Construction~VII.4.6 and~VII.4.7 of~\cite{deligne-rapoport-1973}. In particular, we can apply Corollaire~VII.3.12 to compare $\pi$\nbd adic valuations of the Fourier expansion of~$f$ and~$f |_k\,\ga$ for every place~$\pi$ of~$K$ lying above~$\ell$. Using the notation by Deligne-Rapoport~\cite{deligne-rapoport-1973}, the condition on the image of their $g \in \SL{}(2, \ZZ \slash n)$ in the group~$\SL{}(2, \ZZ \slash p^m)$ translates into our condition that~$\gcd(\ell, N \slash N_\ell) = 1$.
\end{proof}

\subsection{Congruences of modular forms on arithmetic progressions}

We will identify upper triangular matrices~$\ga \in \GL{2}^+(\RR)$ with~$(\ga, \tau \mto \sqrt{\det(\ga)}) \in \Mp{1}(\RR)$.

Fix co-prime positive integers~$M$ and~$N$. Consider an abstract module of weakly holomorphic modular forms~$(V,\phi)$ for some~$\tdGa_0(N)$. Let~$v \in V$ be a $T$\nbd eigenvector, let~$\beta \in \QQ$ such that~$v | T = e(\beta) v$, and assume that~$(V,\phi)$ is defined over a ring that contains~$e(\beta \slash M)$. We define
\begin{gather}
\label{eq:def:hecke-T-eigenvectors}
  \rmT_M(v, \beta)
\;:=\;
  v \otimes
  \sum_{h \,\pmod{M}}
  e\big( - \tfrac{h \beta}{M} \big)\,
  \fraku_{\begin{psmatrix} 1 & h \\ 0 & M \end{psmatrix}}
  \,\in\,
  \rmT_M\,V
\tx{.}
\end{gather}
A brief verification shows that the sum over~$h \,\pmod{M}$ is well-defined and that we have
\begin{gather}
  \rmT_M(v,\beta) \big| T
\;=\;
  e\big(\tfrac{\beta}{M}\big)\, \rmT_M(v,\beta)
\tx{.}
\end{gather}

The image of~$\rmT_M(v,\beta)$ under~$\rmT_M\,\phi$ captures the Fourier coefficients of~$\phi(v)$ on the arithmetic progression~$M \ZZ + \beta$. We have to restrict to abstract modules of weak holomorphic modular forms for groups~$\tdGa_0(N)$, since we have defined Hecke operators at this level of generality.
\begin{proposition}
\label{prop:hecke-T-eigenvectors-fourier-expansion}	
Fix a positive integer~$M$, a rational number~$\beta$, and an abstract module of weakly holomorphic modular forms~$(V,\phi)$ for~$\tdGa_0(N) \subseteq \Mp{1}(\ZZ)$ over a ring that contains~$e(\beta \slash M)$. Fix a $T$\nbd eigenvector~$v \in V$ with eigenvalue~$e(\beta)$. Then we have
\begin{gather}
\label{eq:prop:hecke-T-eigenvectors-fourier-expansion}
  (\rmT_M\phi) \big( \rmT_M(v,\beta) \big)
\;=\;
  M^{1-\frac{k}{2}}
  \sum_{n \in \beta + M \ZZ}
  c(f;\,n)
  e\big( \tfrac{n}{M}\, \tau\big)
\tx{.}
\end{gather}
\end{proposition}
\begin{proof}
The definition of the slash action in~\eqref{eq:slash-action-mp1r} yields
\begin{gather*}
  e(n \tau) \big|_k\, \begin{psmatrix} 1 & h \\ 0 & M \end{psmatrix}
\;=\;
  M^{-\frac{k}{2}}\,
  e\big( \tfrac{h}{M} n \big)\,
  e\big( \tfrac{n}{M}\, \tau\big)
\tx{.}
\end{gather*}
We let $f = \phi(v)$, and observe that the eigenvector equation~$v | T = e(\beta) v$ implies that the Fourier coefficients~$c(f;\,n)$ are supported on~$n \in \ZZ + \beta$. When inserting the definition of~$\rmT_M(v,\beta)$ and the Fourier expansion of~$f$, we find that
\begin{align*}
  (\rmT_M\phi) \big( \rmT_M(v,\beta) \big)
\;&=\;
  \sum_{h \,\pmod{M}}
  e\big( - \tfrac{h \beta}{M} \big)\,
  f \big|_k\, \big( \begin{psmatrix} 1 & h \\ 0 & M \end{psmatrix},\, \tau \mto \sqrt{M} \big)
\\
&=\;
  M^{-\frac{k}{2}}\,
  \sum_{n \in \ZZ + \beta}
  c(f;\,n)
  e\big( \tfrac{n}{M}\, \tau\big)
  \sum_{h \,\pmod{M}}
  e\big( \tfrac{h(n - \beta)}{M} \big)
\\
&=\;
  M^{1-\frac{k}{2}}\,
  \sum_{\substack{n \in \ZZ + \beta\\n-\beta \equiv 0 \,\pmod{M}}}
  c(f;\,n)
  e\big( \tfrac{n}{M}\, \tau\big)
\tx{.}
\end{align*}
\end{proof}

As a consequence of Proposition~\ref{prop:hecke-T-eigenvectors-fourier-expansion}, Ramanujan-type congruences for a modular form~$f$ on arithmetic projections~$M \ZZ + \beta$ correspond to vectors $\rmT_M(v,\beta)$ in the $\ell$\nbd kernel of $\rmT_M\, \CC\,f$. The next corollary describes this connection in a precise way.
\begin{corollary}
\label{cor:congruence-on-arithmetic-progression-hecke-T-eigenvector}
Fix a positive integer~$M$, a rational number~$\beta$, a number field $K \subset \CC$ that contains~$e(\beta \slash M)$, and an ideal~$\ell \subseteq \OK$ that is co-prime to~$M$. Consider an abstract module of weakly holomorphic modular forms~$(V,\phi)$ for~$\Ga \subseteq \Mp{1}(\ZZ)$ over~$\OKell$. Fix a $T$\nbd eigenvector~$v \in V$ with eigenvalue~$e(\beta)$. Then the Ramanujan-type congruence
\begin{gather*}
\label{eq:cor:congruence-on-arithmetic-progression-hecke-T-eigenvector:arithmetic-progression}
  \forall n \in \ZZ :\,
  c(f;\,M n + \beta) \equiv 0 \;\pmod{\ell}
\end{gather*}
for the modular form~$f = \phi(v)$ is equivalent to
\begin{gather*}
\label{eq:cor:congruence-on-arithmetic-progression-hecke-T-eigenvector:hecke}
  \rmT_M(v,\beta)
\in
  \ellker\big( \rmT_M (V,\phi) \big)
\tx{.}
\end{gather*}
\end{corollary}

The next proposition allows us to focus on the case of prime powers~$M$ when analyzing $\ell$\nbd kernels and Ramanujan-type congruences in the second part of this series of papers.
\begin{proposition}
\label{prop:composition-of-hecke-operators}
Fix co-prime positive integers~$M_1$ and~$M_2$, a rational number~$\beta$, and an abstract module of weakly holomorphic modular forms~$(V,\phi)$ for~$\tdGa_0(N) \subseteq \Mp{1}(\ZZ)$ over a ring that contains~$e(\beta \slash M_1 M_2)$. Assume that~$T \in \Ga$ and fix a $T$\nbd eigen\-vector~$v \in V$ with eigenvalue~$e(\beta)$. Then under the isomorphism in~\eqref{eq:hecke-operator-decomposition} we have
\begin{gather*}
  \rmT_{M_1}\big( \rmT_{M_2}(v, \beta), \tfrac{\beta}{M_2} \big)
\lmto
  \rmT_{M_1 M_2}(v, \beta)
\tx{.}
\end{gather*}

In particular, in the case that $(V,\phi)$ is defined over~$\OKell$ for a number field $K \subset \CC$ that contains~$e(\beta \slash M_1 M_2)$ with an ideal~$\ell \subseteq \OK$ that is co-prime to~$M_1 M_2$, we have that
\begin{gather*}
  \rmT_{M_1}\big( \rmT_{M_2}(v, \beta), \tfrac{\beta}{M_2} \big)
\in
  \ellker(\rmT_{M_1}\, \rmT_{M_2}\, V)
\end{gather*}
is equivalent to
\begin{gather*}
  \rmT_{M_1 M_2}(v, \beta)
\in
  \ellker(\rmT_{M_1 M_2}\, V)
\tx{.}
\end{gather*}
\end{proposition}
\begin{proof}
We insert the defining expression in~\eqref{eq:def:hecke-T-eigenvectors}, and in the next to last step rewrite the integer~$h_1 + M_1 h_2$, which is well-defined modulo~$M_1 M_2$, as $h \,\pmod{M_1 M_2}$:
\begin{alignat*}{2}
  \rmT_{M_1}\big( \rmT_{M_2}(v, \beta), \tfrac{\beta}{M_2} \big)
&\;=\;&&
  \rmT_{M_2}(v, \beta)
  \,\otimes\,
  \sum_{h_1 \,\pmod{M_1}}
  e\big(-\tfrac{h_1 \beta}{M_1 M_2}\big)\,
  \fraku_{\begin{psmatrix} 1 & h_1 \\ 0 & M_1 \end{psmatrix}}
\\
&\;=\;&&
  v
  \,\otimes\,
  \sum_{h_2 \,\pmod{M_2}}
  e\big(-\tfrac{h_2 \beta}{M_2}\big)\,
  \fraku_{\begin{psmatrix} 1 & h_2 \\ 0 & M_2 \end{psmatrix}}
  \,\otimes\,
  \sum_{h_1 \,\pmod{M_1}}
  e\big(-\tfrac{h_1 \beta}{M_1 M_2}\big)\,
  \fraku_{\begin{psmatrix} 1 & h_1 \\ 0 & M_1 \end{psmatrix}}
\\
&{}\lmto{}&&
  v
  \,\otimes\,
  \sum_{\substack{h_1 \,\pmod{M_1}\\h_2 \,\pmod{M_2}}}
  e\big(-\tfrac{h_2 \beta}{M_2}
        -\tfrac{h_1 \beta}{M_1 M_2}\big)\,
  \fraku_{\begin{psmatrix} 1 & h_2 \\ 0 & M_2 \end{psmatrix}
          \begin{psmatrix} 1 & h_1 \\ 0 & M_1 \end{psmatrix}}
\\
&\;=\;&&
  v
  \,\otimes\,
  \sum_{\substack{h_1 \,\pmod{M_1}\\h_2 \,\pmod{M_2}}}
  e\big(-\tfrac{(h_1 + M_1 h_2)\beta}{M_1 M_2}\big)\,
  \fraku_{\begin{psmatrix} 1 & h_1 + M_1 h_2 \\ 0 & M_1 M_2 \end{psmatrix}}
\\
&\;=\;&&
  v
  \,\otimes\,
  \sum_{h \,\pmod{M_1 M_2}}
  e\big(-\tfrac{h \beta}{M_1 M_2}\big)\,
  \fraku_{\begin{psmatrix} 1 & h \\ 0 & M_1 M_2 \end{psmatrix}}
\;=\;
  \rmT_{M_1 M_2}( v, \beta)
\tx{.}
\end{alignat*}
\end{proof}

\subsection{Nontrivial congruences}
\label{ssec:nontrivial-congruences}

It is hard in many cases to decipher which vectors are contained in the $\ell$\nbd kernel. We saw in Corollary~\ref{cor:congruence-on-arithmetic-progression-hecke-T-eigenvector} that a general answer would include a description of all Ramanujan-type congruences. Theorem~\ref{thm:ell-kernel-representation} allows us to endow ``extra'' congruences with more structure.

More precisely, let~$V_\ell = V$ if $V$ is defined over~$\OKell$ or more generally let~$V_\ell \subseteq V$ be an~$\OKell[\tdGa_0(N_\ell) \cap \Ga]$\nbd structure that is compatible with Fourier expansions in the sense that $\rmfe( \phi( V_\ell )) \subseteq \rmFE( \OKell )$. Here, we maintain the assumptions on~$(V,\phi)$ from Theorem~\ref{thm:ell-kernel-representation} and let~$N_\ell$ be as in~\eqref{eq:thm:ell-kernel-representation:Nell}. In applications, $V_\ell$ will consist of vectors whose realization a priori has $\ell$\nbd integral Fourier coefficients.

Assume that~$\ell \subset \OK$ is a prime ideal with associated residue field~$\Fell$. We consider the $\Fell[\Ga]$\nbd module
\begin{gather*}
  \ov{V}_\ell
\;:=\;
  V_\ell \slash \ell V_\ell
\end{gather*}
and the following right-module for $\Fell[\tdGa_0(N_\ell) \cap \Ga]$:
\begin{gather}
\label{eq:def:ovell-kernel}
\begin{aligned}
  \ovellker\big( (V,\phi) \big)
\;:=\;&
  \big( \ellker\big( (V,\phi) \big) \cap V_\ell \big)
  \,\big\slash\,
  \big( \ell V_\ell \cap \ellker\big( (V,\phi) \big) \big)
\lhra
  \ov{V}_\ell
\tx{.}
\end{aligned}
\end{gather}
If, as suggested in the preceding paragraph, $V_\ell$ comprises all a priori known vectors with $\ell$\nbd integral Fourier expansion, then~$\ov{V}_\ell$ tracks modular forms up to congruence modulo~$\ell$ and any nontrivial vector in~$\ovellker((V,\phi))$ yields additional congruences.

\begin{example}
\label{ex:ovell-kernel-half-integral-weight:atkin}
Isomorphism~\eqref{eq:def:ovell-kernel} can be fruitfully employed to endow Ramanujan-type congruences for half integral weight modular forms with extra structure. As an example, we investigate the classical congruence
\begin{gather*}
  \forall n \in \ZZ \,:\,
  p(11^3\, 13 n + 237)
  \equiv
  0
  \;\pmod{13}
\end{gather*}
of the partition function~$p(n)$ discovered by Atkin~\cite{atkin-1968}. He observed that it is one of 30~congruences modulo~$\ell = 13$ on arithmetic progressions~$11^3\, 13 \ZZ + b$. The classes of integers~$b$ that appear are those for which~$b - \frac{1}{24}$ lies in the same square class modulo~$11^3\, 13$ as~$237 - \frac{1}{24}$. A connection among congruences for the partition function following the same pattern was established by Radu~\cite{radu-2012}, and will be extended to arbitrary modular forms in the second paper of this series.

The $\ell$\nbd kernel yields slightly more information on Atkin's congruences. Specifically, there is a cusp form~$f$ of weight~$\frac{\ell^2}{2} - 1$ such that
\begin{gather*}
  c\big( f, \tfrac{n}{24} \big)
\equiv
  p\big( \tfrac{n+1}{24} \big)
  \;\pmod{\ell}
\tx{,}\quad
  \tx{if }\left(\mfrac{-n}{\ell}\right) = -1
\tx{;}\qquad
  c\big( f, \tfrac{n}{24} \big)
\equiv
  0
  \;\pmod{\ell}
\tx{,}\quad
  \tx{otherwise.}
\end{gather*}

Now~$f$ is a modular form for the character~$\ov{\chi}_\eta$ of~$\Mp{1}(\ZZ)$, where~$\chi_\eta$ is the character for the Dedekind~$\eta$\nbd function. We conclude that the abstract space of modular forms~$\CC\,f$ is isomorphic to~$\chi_\eta$ as a representation of~$\Mp{1}(\ZZ)$.

Let~$M = 11^3$, and observe that Atkin's congruence corresponds to the vector
\begin{gather*}
 \rmT_M(f,\beta)
\in
  \ellker \big( \rmT_M(\CC\,f) \big)
\tx{,}\quad
  \beta \in \ZZ - \tfrac{1}{24}
\tx{,}\,
  \beta \equiv 7 \cdot 11^2 \,\pmod{11^3}
\tx{.}
\end{gather*}
The representation generated by this vector is~$6$\nbd dimensional, and by Theorem~\ref{thm:ell-kernel-representation} it is a subrepresentation of the~$\ell$\nbd kernel of~$\rmT_M\,\CC\,f$. The five congruence classes modulo~$11$ that are non-squares yield five vectors~$\rmT_M(f, \beta')$ for various~$\beta'$ in this subrepresentation, and each of them gives one of Atkin's Ramanujan-type congruences. We could combine this with a computation of $\rmT_{13}\,\CC\,f$ to explain all~30 of Atkin's Ra\-ma\-nu\-jan-type congruences.

In addition, to the $5$\nbd dimensional space that the above vectors~$\rmT_M(f, \beta')$ span, we find another vector that corresponds to a congruence of~$f$ under a specific Hecke-like operator.
\end{example}

\begin{example}
\label{ex:ovell-kernel-half-integral-weight:atkin-hypothetical}
We continue Example~\ref{ex:ovell-kernel-half-integral-weight:atkin}. Assume that we have a weakly holomorphic modular form~$f$ of weight~$\frac{13^2}{2} - 1$ for~$\ov\chi_\eta$ and the congruence
\begin{gather*}
  \forall n \in \ZZ \,:\,
  c(f; 11^3\, 13 n + 1810 - \tfrac{1}{24})
  \equiv
  0
  \;\pmod{13}
\tx{.}
\end{gather*}
The resulting subrepresentation of~$\ovellker (\rmT_M(\CC\,f))$ would be~$6$\nbd dimensional, subsuming five Ramanujan-type congruences and one congruence with respect to a Hecke-like operator. This stands in perfect analogy with the computation in Example~\ref{ex:ovell-kernel-half-integral-weight:atkin}.

When joining these calculations
one of the missing square classes modulo~$13^3$ appears. Assume that~$f$ also satisfies the analogues of Atkin's partition congruences in Example~\ref{ex:ovell-kernel-half-integral-weight:atkin}. This would result in a~$12$\nbd dimensional subrepresentation of the~$\ell$\nbd kernel. This~$12$\nbd dimensional representation contains the vector~$\rmT_M(f, 11^4)$. Rephrased in classical terms, a congruence on the arithmetic progression~$11^3\, 13 \ZZ + 237 - \frac{1}{24}$ when coexisting with a congruence on~$11^3\, 13 \ZZ + 1810 - \frac{1}{24}$, already implies a congruence on the much larger progression~$11^2\, 13 \ZZ + 237 - \frac{1}{24}$. We suspect that this is a general phenomenon for Ramanujan-type congruences for, say, the partition function on~$q^3 \ell \ZZ + b$ for prime numbers~$q$. It stands in stark contrast to the congruences on~$q^4 \ell \ZZ + b$, which we examine next.
\end{example}

\begin{example}
\label{ex:ovell-kernel-half-integral-weight:ono}
Ono~\cite{ono-2000} and Ahlgren-Ono~\cite{ahlgren-ono-2001} discovered that there are infinitely many primes~$q$ and integers~$b$ for which we have Ramanujan-type congruences with gaps
\begin{gather*}
  \forall n \in \ZZ \setminus q \ZZ \,:\,
  p(q^3\, \ell n + b)
  \equiv
  0
  \;\pmod{\ell}
\tx{.}
\end{gather*}
They lead to $q-1$ distinct Ramanujan-type congruences
\begin{gather*}
  \forall m \in \ZZ \setminus q \ZZ\;
  \forall n \in \ZZ \,:\,
  p(q^4\, \ell n + b + q^3\, \ell m)
  \equiv
  0
  \;\pmod{\ell}
\tx{.}
\end{gather*}
These are grouped in exactly the way that we suggested in Example~\ref{ex:ovell-kernel-half-integral-weight:atkin-hypothetical} should not happen for Ramanujan-type congruences on~$q^3\,\ell \ZZ + b$. A computation reveals that the~$\ell$\nbd kernel prescribes exactly opposite behavior of Ramanujan-type congruences on~$q^3\,\ell \ZZ + b$ and~$q^4\,\ell \ZZ + b$.

For simplicity, and to keep dimensions of~$\rmT_M\,\CC\,f$ sufficiently low (this will be a representation of dimension~$15\,972$), we will investigate the case~$q=11$ and~$M=11^4$ for a modular form~$f$ as in Examples~\ref{ex:ovell-kernel-half-integral-weight:atkin} or~\ref{ex:ovell-kernel-half-integral-weight:atkin-hypothetical}. Let us assume that the~$\ell$\nbd kernel contains, for example, the vector~$\rmT_M(f,\beta)$ with $\beta \in \ZZ - \frac{1}{24}$ and~$\beta \equiv 11^3 \,\pmod{11^4}$. As opposed to Example~\ref{ex:ovell-kernel-half-integral-weight:atkin}, this generates an~$11$\nbd dimensional representation, which contains~$\rmT_M(f,\beta')$ for all $\beta' \in \ZZ - \frac{1}{24}$ with~$11^3 \isdiv \beta$ and~$11^4 \nisdiv \beta$, i.e.\@ $11^3 \| \beta$.

This suggests that Ramanujan-type congruences on~$q^4 \ell \ZZ + b$ imply Ramanujan-type congruences with gap on~$q^3 \ell (\ZZ \setminus q \ZZ) + q^4 b'$ with~$q^4 b' \equiv b \,\pmod{\ell}$. An analogous statement for integral weights is available in Theorem~\ref{thm:structure-of-ramanujan-type-congruences}.
\end{example}

\begin{example}
\label{ex:ovell-kernel-half-integral-weight:ono-hypothetical}
To complement our discussion in Example~\ref{ex:ovell-kernel-half-integral-weight:ono} we examine the consequences of a congruence
\begin{gather*}
  \forall n \in \ZZ \,:\,
  c(f; q^4\, \ell n + b - \tfrac{1}{24})
\end{gather*}
where $q^4 \isdiv (b - \frac{1}{24})$.

We restrict to the case of~$q = 11$. The associated vector in the~$\ell$\nbd kernel of the abstract space of modular forms~$\rmT_M\,\CC\,f$ generates a~$12$\nbd dimensional subrepresentation, which contains the one that we encountered in Example~\ref{ex:ovell-kernel-half-integral-weight:ono}. In other words, it implies by virtue of Theorem~\ref{thm:ell-kernel-representation} alone that there is a congruence
\begin{gather*}
  \forall n \in \ZZ \,:\,
  c(f; q^3\, \ell n + b - \tfrac{1}{24})
\tx{,}
\end{gather*}
which then by the calculations in Examples~\ref{ex:ovell-kernel-half-integral-weight:atkin} and~\ref{ex:ovell-kernel-half-integral-weight:atkin-hypothetical} implies a congruence 
\begin{gather*}
  \forall n \in \ZZ \,:\,
  c(f; q^2\, \ell n + b - \tfrac{1}{24})
\tx{.}
\end{gather*}
Continuing this chain of calculations one shows that this would eventually imply
\begin{gather*}
  \forall n \in \ZZ \,:\,
  c(f; \ell n + b - \tfrac{1}{24})
\tx{.}
\end{gather*}

While our calculation is limited to~$q = 11$, it provides evidence that Ramanujan-type congruence for the partition function on~$q^4 \ell n + b$ cannot have~$q^4 \isdiv (b - \frac{1}{24})$.
\end{example}

\begin{example}
\label{ex:ovell-kernel-integral-weight}
In the case of level~$1$, Theorem~\ref{thm:structure-of-ramanujan-type-congruences} almost exhausts the information provided by~\eqref{eq:def:ovell-kernel}. For instance, consider the Ramanujan~$\Delta$ function and the associated abstract module of modular forms~$\CC \Delta$. Let~$\ell = 7$, for which~$\Delta$ is known to exhibit a~$\rmU_\ell$\nbd congruence. In other words, we have
\begin{gather*}
  \forall n \in \ZZ :\,
  c(\Delta;\,\ell n) \equiv 0 \;\pmod{\ell}
\tx{.}
\end{gather*}
In this situation, the~$\ell$\nbd kernel of~$\rmT_\ell\,\CC\Delta$ does not yield any additional information.

Specifically, let~$K$ be the~$7$\thdash\ cyclotomic field. We have a natural~$\OKell$\nbd structure
\begin{alignat*}{2}
  V_\ell
\;&:=\;&&
  \ell^{-6}\,
  \OKell \Delta \otimes \begin{psmatrix} \ell & 0 \\ 0 & 1 \end{psmatrix}
  \,+\,
  \ell^5\, 
  \OKell \big\{
  \rmT_\ell(\Delta, \beta) \,:\,
  \beta \,\pmod{\ell}
  \big\}
\\
&{}\subset{}&&
  \ell^{-6}\,
  \OKell \Delta \otimes \begin{psmatrix} \ell & 0 \\ 0 & 1 \end{psmatrix}
  \,+\,
  \ell^6\, 
  \OKell \big\{
  \Delta \otimes \begin{psmatrix} 1 & h \\ 0 & \ell \end{psmatrix} \,:\,
  h \,\pmod{\ell}
  \big\}
\tx{.}
\end{alignat*}
The~$\rmU_\ell$\nbd congruence yields
\begin{gather*}
  \ell^5\, 
  \rmT_\ell(\Delta,0)
\in
  \ellker( \CC \Delta )
\tx{,}
\end{gather*}
and therefore we have a one-dimensional, trivial representation of $\Ga_0(\ell)$:
\begin{gather*}
  \ovellker( \CC \Delta )
\;=\;
  \Fell \big\{
  \ell^5\, 
  \rmT_\ell(\Delta,0)
  \big\}
\tx{,}
\end{gather*}
where on the right hand side the vector~$\ell^5\, \rmT_\ell(\Delta,0)$ is understood as a formal generator.
\end{example}

\section{The structure of Ramanujan-type congruences}
\label{sec:structure}

In this section, we provide our main structure theorem for Ramanujan-type congruences. Most importantly, it asserts that Ramanujan-type congruence modulo~$\ell$ for modular forms of integral weight, if~$\ell \nisdiv 2$, are associated with the union of two square-classes. In other words, they yield Ramanujan-type congruences with gap.

\begin{theorem}
\label{thm:structure-of-ramanujan-type-congruences}
Fix a number field $K \subset \CC$ and a prime ideal~$\ell \subset \OK$. Let~$f$ be a weakly holomorphic modular form of integral weight for a Dirichlet character modulo~$N$ with Fourier coefficients~$c(f;\,n) \in \OKell$. Assume that~$f$ satisfies the Ramanujan-type congruence
\begin{gather}
\label{eq:thm:structure-of-ramanujan-type-congruences:congruence}
  \forall n \in \ZZ \,:\, c(f;\,M n + \beta) \equiv 0 \;\pmod{\ell}
\end{gather}
for some positive integer~$M$ and some integer~$\beta$. Consider a prime~$p \isdiv M$ that is co-prime to~$\ell N$, and factor~$M$ as~$M_p M_p^\#$ with a~$p$\nbd power~$M_p$ and~$M_p^\#$ co-prime to~$p$.

\begin{enumeratearabic}
\item
\label{it:thm:structure-of-ramanujan-type-congruences:ramanujan-type-congruence-with-gap}
If\/~$\ell \nisdiv 2$, then we have the Ramanujan-type congruence with gap
\begin{gather}
\label{eq:thm:structure-of-ramanujan-type-congruences:congruence-both-square-classes}
  \forall n \in \ZZ \setminus p \ZZ \,:\,
  c(f;\,(M \slash p) n + M_p \beta') \equiv 0 \;\pmod{\ell}
\quad
\tx{with }
  M_p \beta' \equiv \beta \;\pmod{M_p^\#}
\tx{.}
\end{gather}

\item
\label{it:thm:structure-of-ramanujan-type-congruences:square-free}
We have
\begin{gather}
\label{eq:thm:structure-of-ramanujan-type-congruences:square-free}
  \forall n \in \ZZ \,:\,
  c(f;\,M' n + \beta) \equiv 0 \;\pmod{\ell}
\quad
\tx{with }
  M' = \gcd(M, M_{\rms\rmf} \beta )
\tx{,}
\end{gather}
where~$M_{\rms\rmf}$ is the largest square-free divisor of\/~$M$.

\item
\label{it:thm:structure-of-ramanujan-type-congruences:beta-equiv-zero}
If\/~$M_p \isdiv \beta$, then we have the Ramanujan-type congruence
\begin{gather}
\label{eq:thm:structure-of-ramanujan-type-congruences:congruence-remove-prime:Up}
  \forall n \in \ZZ \,:\,
  c(f;\,M_p^\# n + \beta) \equiv 0 \;\pmod{\ell}
\tx{.}
\end{gather}

If~$p^2 \nisdiv M$ and~$\ell \nisdiv 2$, then we have the congruence
\begin{gather}
\label{eq:thm:structure-of-ramanujan-type-congruences:congruence-remove-prime:squarefree}
  \forall n \in \ZZ, M_p^\# n + \beta \ne 0 \,:\,
  c(f;\,M_p^\# n + \beta) \equiv 0 \;\pmod{\ell}
\tx{.}
\end{gather}
\end{enumeratearabic}
\end{theorem}

\begin{remark}
If~$\ell$ is an odd rational prime and~$f$ is a cusp form, the congruence in~\eqref{eq:thm:structure-of-ramanujan-type-congruences:congruence-remove-prime:squarefree} means that we can remove all primes~$p \ne \ell$ from~$M$ that do not appear in the level~$N$ and for which~$\beta$ has highest possible $p$\nbd divisibility or for which~$p$ exactly divides~$M$. Compare this with Proposition~\ref{prop:integral-weight-p-congruences-treneer}, which goes back to Treneer~\cite{treneer-2006}, for examples of Ra\-ma\-nu\-jan-type congruences where~$p^2$ exactly divides~$M$.
\end{remark}

We postpone the proof of Theorem~\ref{thm:structure-of-ramanujan-type-congruences} in favor of two examples illustrating the limits of statements like the ones in Theorem~\ref{thm:structure-of-ramanujan-type-congruences}. They show that the condition that~$p$ be co-prime to~$\ell N$ in~\eqref{eq:thm:structure-of-ramanujan-type-congruences:congruence-both-square-classes} and~\eqref{eq:thm:structure-of-ramanujan-type-congruences:congruence-remove-prime:Up} is ``sharp''.

\begin{example}
It is straightforward to check with any suitable computer algebra system that
\begin{gather*}
  \forall n \in \ZZ \,:\, c(\Delta;\,\ell n) \equiv 0 \;\pmod{\ell}
\end{gather*}
for~$\ell \in \{2, 3, 5, 7, 2411\}$ (there is one more known for~$\ell = 7758337633$, which is slightly harder to verify). In this setting, $M = \ell$, $N = 1$, $\beta = 0$, $p = \ell$, but~$M_\ell = \ell$ does not allow us to strengthen the given congruence. In particular, the condition~$\ell \nisdiv p$ in~\eqref{eq:thm:structure-of-ramanujan-type-congruences:congruence-both-square-classes} and~\eqref{eq:thm:structure-of-ramanujan-type-congruences:congruence-remove-prime:Up} is necessary.
\end{example}

\begin{example}
\label{ex:hecke-operators:oldforms-twists}
The condition that~$p \nisdiv N$ in~\eqref{eq:thm:structure-of-ramanujan-type-congruences:congruence-both-square-classes} and~\eqref{eq:thm:structure-of-ramanujan-type-congruences:congruence-remove-prime:Up} is also necessary. Counter examples arise from the Ramanujan-type congruences trivially satisfied by oldforms and twists of modular forms by Dirichlet characters.

Specifically, consider a modular form~$f$ of level~$1$. Then for any prime~$p$ the oldform~$f_p(\tau) := f(p \tau)$ satisfies the Ramanujan-type congruences
\begin{gather*}
  \forall n \in \ZZ \,:\, c(f_p;\,p n + \beta) \equiv 0 \;\pmod{\ell}
\end{gather*}
for all integers~$\beta$ co-prime to~$p$. Similarly, the twist $f_\chi(\tau) := \sum_n \chi(n) c(f;n) e(n \tau)$ of~$f$ by a Dirichlet character~$\chi$ modulo~$p$ satisfies the Ramanujan-type congruence
\begin{gather*}
  \forall n \in \ZZ \,:\, c(f_\chi;\,p n) \equiv 0 \;\pmod{\ell}
\tx{.}
\end{gather*}
\end{example}

We need the following proposition in the proof of Theorem~\ref{thm:structure-of-ramanujan-type-congruences}. While we suspect that it is known to some experts, we could not find it in the literature. Its proof is straightforward and could in fact be carried out without reference to abstract spaces of modular forms, relying directly on the work of Deligne-Rapoport~\cite{deligne-rapoport-1973} instead.
\begin{proposition}
\label{prop:congruences-square-free-quotient}
Fix a number field $K \subset \CC$ and an ideal~$\ell \subseteq \OK$. Let~$f$ be a weakly holomorphic modular form of integral weight for~$\Ga_0(N) \cap \Ga(N_1)$ with Fourier coefficients~$c(f;\,n) \in \OKell$  supported on~$n \in \ZZ + \beta$ for some~$\beta \in \QQ$. Assume that~$f$ satisfies the Ramanujan-type congruence
\begin{gather*}
  \forall n \in \ZZ \,:\, c(f;\,M n + \beta) \equiv 0 \;\pmod{\ell}
\end{gather*}
for some positive integer~$M$ that is co-prime to~$N_1$. Then~$\beta$ is $M$\nbd integral and for all $M$\nbd integral~$\beta' \in N_1 \ZZ + \beta$ in the same square-class modulo~$M$ as~$\beta$, we have
\begin{gather*}
  \forall n \in \ZZ \,:\, c(f;\,M n + \beta') \equiv 0 \;\pmod{\ell}
\tx{.}
\end{gather*}

In particular, we have
\begin{gather*}
  \forall n \in \ZZ \,:\, c(f;\,M' n + \beta) \equiv 0 \;\pmod{\ell}
\tx{,}
\end{gather*}
where~$M' = \gcd(M, M_{\rms\rmf} N_1 \beta )$ and $M_{\rms\rmf} = \gcd(8,M)\, \prod_{p} p$ with $p$~running through all odd prime divisors of~$M$.
\end{proposition}
\begin{proof}
The fact that~$\beta$ is~$M$\nbd integral follows directly from~$\gcd(M,N_1) = 1$. Given~$\beta'$ as in the statement, there is~$u \in \ZZ$ invertible modulo~$M N$ and with~$u \equiv 1 \,\pmod{N_1}$ such that~$u^2 \beta \equiv \beta' \,\pmod{M}$. Let~$\ov{u} \in \ZZ$ denote the inverse of~$u$ modulo~$M N N_1$. 

To establish the first part of the proposition, we will investigate the vectors~$\rmT_M(f, \beta)$ in the abstract module of modular forms~$\rmT_M\, \Ind\,\CC\,f$. Observe that there is an element~$\ga$ in~$\Ga_0(M N) \cap \Ga(N_1)$ that is congruent to $\begin{psmatrix} u & 0 \\ 0 & \ov{u} \end{psmatrix}$ modulo~$M N N_1$. For any such~$\ga$ and~$h \in \ZZ$, we have
\begin{gather*}
  \ga^{-1}\, \begin{pmatrix} 1 & h \\ 0 & M \end{pmatrix} \,\ga
\;\equiv\;
  \begin{pmatrix} 1 & \ov{u}^2 h \\ 0 & M \end{pmatrix}
  \;\pmod{M N N_1}
\tx{.}
\end{gather*}
In particular, we have
\begin{gather*}
   \begin{pmatrix} 1 & h \\ 0 & M \end{pmatrix}
   \,\ga\,
   \begin{pmatrix} 1 & \ov{u}^2 h \\ 0 & M \end{pmatrix}^{-1}
\;\in\;
  \Ga_0(N) \cap \Ga(N_1)
\tx{.}
\end{gather*}
As a consequence, conjugating~$\rmT_M(f, \beta)$ by~$\ga$ yields~$\rmT_M(f, \beta')$. We can now apply Corollary~\ref{cor:congruence-on-arithmetic-progression-hecke-T-eigenvector} to finish the proof of the first statement.

To prove the second part, it suffices to show that the coset $N_1 \beta (\ZZ \slash M' \ZZ)$ is contained in $N_1 \beta (\ZZ \slash M \ZZ)^{\times\,2}$. We can and will replace~$M$, $M'$, and~$N_1 \beta$ by their quotients by~$\gcd(M,N_1 \beta)$. In particular, we assume that~$M$ and~$N_1 \beta$ are co-prime. Further, we replace~$N_1 \beta$ by~$\beta$ to ease our notation. 

Now the proposition is reduced to the statement that for odd primes~$p$ the square-class of a~$\beta \in \ZZ$ co-prime to~$p$ depends only on~$\beta \,\pmod{p}$, and for~$p = 2$ is depends only on~$\beta \,\pmod{8}$.
\end{proof}

Before proceeding to the proof of Theorem~\ref{thm:structure-of-ramanujan-type-congruences}, we state one more proposition, whose proof can be conducted with little extra work. Despite its fairly technical traits, we suggest to compare it with Theorems~\ref{mainthm:ramanujan-type-implies-hecke} and~\ref{thm:ramanujan-type-implies-hecke}, which only hold for cusp forms as opposed to weakly holomorphic modular forms. The statement of Proposition~\ref{prop:ramanujan-type-implies-hecke-like} requires the following modification of the usual Hecke operators. Effectively, their action differs from the usual ones by twists with values of a Dirichlet character. Given co-prime positive integers~$M$ and~$N$, we set
\begin{gather}
\label{eq:def:classical-hecke-unnormalized}
  \rmTcl_{M,MN}
\;:=\;
  \sum_{\substack{a d = M\\b \pmod{d}}} \ga_{a,d,b}
\quad\tx{and}\quad
  \rmU_{M,MN}
\;:=\;
  \sum_{\substack{b \pmod{M}}} \ga_{1,M,b}
\tx{,}
\end{gather}
where
\begin{gather*}
  \ga_{a,d,b} \in \Mat{2}(\ZZ)
\tx{,}\;
  \det(\ga_{a,d,b}) = M
\tx{,}\;
  \ga_{a,d,b} \equiv \begin{psmatrix} a & b \\ 0 & d \end{psmatrix} \;\pmod{M}
\tx{,}\;
  \ga_{a,d,b} \equiv \begin{psmatrix} 1 & 0 \\ 0 & 1 \end{psmatrix} \;\pmod{N}
\tx{.}
\end{gather*}

\begin{proposition}
\label{prop:ramanujan-type-implies-hecke-like}
Assume that~$K$, $\ell$, $f$, $N$, $M$, and~$\beta$ are as in Theorem~\ref{thm:structure-of-ramanujan-type-congruences}. Factor~$M$ as $M_{\ell N} M^\#_{\ell N}$, where $M_{\ell N}$ is the largest divisor of~$M$ such that no prime~$p \isdiv M_{\ell N}$ is co-prime to~$\ell N$. Consider the weakly holomorphic modular form with Fourier expansion
\begin{gather*}
  f_{\ell N}(\tau)
\;:=\;
  \sum_{\substack{\beta' \in \beta (\ZZ \slash M_{\ell N} \ZZ)^{\times\,2}\\n \in M_{\ell N} \ZZ + \beta'}}
  c(f;\,n) e(n \tau)
\tx{.}
\end{gather*}
With the operators~$\rmTcl$ and~$\rmU$ from~\eqref{eq:def:classical-hecke-unnormalized}, we have
\begin{gather}
\label{eq:thm:structure-of-ramanujan-type-congruences:congruence-hecke}
  f_{\ell N} \Big|
  \prod_{\substack{p \tx{ prime}\\p \isdiv M_{\ell N}^\#}}
  \big(
  (p+1) \rmU_{M_p,MN} - p \rmTcl_{M_p,MN} + p \rmTcl_{M_p \slash p^2,\, MN \slash p^2}
  \big)
\;\equiv\;
  0 \;\pmod{\ell}
\tx{,}
\end{gather}
where we suppress the last summand if~$p^2 \nisdiv M_p$.
\end{proposition}

\subsection{Proofs of Theorem~\ref{thm:structure-of-ramanujan-type-congruences} and Proposition~\ref{prop:ramanujan-type-implies-hecke-like}}

This whole section will be occupied by the proofs of Theorem~\ref{thm:structure-of-ramanujan-type-congruences} and Proposition~\ref{prop:ramanujan-type-implies-hecke-like}. We start by providing a general setup and then give proofs for individual statements.

We can and will replace~$f$ by~$f_{\ell N}$ and may hence assume that~$M$ is co-prime to~$\ell N$. Proposition~\ref{prop:congruences-square-free-quotient} allows us to assume that~$\beta \in \ZZ \cap (M \slash p) \ZZ_p$ if~$p \isdiv M$ is odd and assume that~$\beta \in \ZZ \cap (M \slash 8) \ZZ_2$ if~$p = 2 \isdiv M$.

Consider the abstract module of modular forms~$(V,\phi) = \OKell\, f$ for~$\Ga_0(N)$ whose kernel contains~$\Ga(N)$. Observe that the module~$V$ is associated with a subgroup of~$\SL{2}(\ZZ)$ as opposed to~$\Mp{1}(\ZZ)$, since the weight of~$f$ is integral. By Corollary~\ref{cor:congruence-on-arithmetic-progression-hecke-T-eigenvector}, the Ramanujan-type congruence satisfied by~$f$ yields the vector
\begin{gather*}
  \rmT_M( f, \beta )
\in
  \ellker\big( \rmT_M\,V\big)
\tx{.}
\end{gather*}

By~\eqref{eq:hecke-to-induced-subgroup} we have the inclusion
\begin{gather}
\label{eq:thm:structure-of-ramanujan-type-congruences:hecke-to-induced}
  \Ind_{\Ga_0(M N)}^{\Ga_0(N)}\, \psi_M^\ast(V)
\lhra
  \rmT_M\,V
\tx{,}
\end{gather}
whose image contains~$\rmT_M(f,\beta)$. Since the kernel of~$\rmT_M\,V$ contains~$\Ga(M N)$, since we have an isomorphism of~$\SL{2}(\ZZ \slash M N)$ and~$\SL{2}(\ZZ \slash M) \times \SL{2}(\ZZ \slash N)$ by the Chinese remainder theorem, and the since homomorphism $\Ga_0(N) \thra \SL{2}(\ZZ \slash M \ZZ)$ arising from this isomorphism is surjective, we can and will view~$\rmT_M\,V$ as an~$\OKell[ \SL{2}(\ZZ \slash M \ZZ)]$-module.

Write~$\psi^\ast_M(f)$ for the pullback of~$f \in V$ to emphasize the difference. To ease notation, we will identify domain and co-domain of the isomorphism
\begin{gather}
\label{eq:thm:structure-of-ramanujan-type-congruences:p1-representation}
\begin{aligned}
  \Ind_{\Ga_0(M N)}^{\Ga_0(N)}\,\psi^\ast_M(V)
\;&\lra\;
  \OKell\big\{ \bbP^1(\ZZ \slash M \ZZ) \big\}
\tx{,}
\\
  \psi^\ast_M(f) \otimes \fraku_\ga
&\lmto
  (c:d) \;\pmod{M}
\tx{,}\quad
  \ga \in \Ga_0(N),\,
  \ga = \begin{psmatrix} a & b \\ c & d \end{psmatrix}
\tx{.}
\end{aligned}
\end{gather}
The co-domain is the free $\OKell$\nbd module whose basis corresponds to the elements of the projective line over~$\ZZ \slash M \ZZ$. The $\SL{2}(\ZZ \slash M \ZZ)$\nbd action on~$\bbP^1(\ZZ \slash M \ZZ)$ arises from its action on row vectors. We record that when combining~\eqref{eq:thm:structure-of-ramanujan-type-congruences:hecke-to-induced} and~\eqref{eq:thm:structure-of-ramanujan-type-congruences:p1-representation}, we have
\begin{alignat*}{7}
  \rmT_M(f, \beta)
&{}\;=\;{}&&
  f
  &&{}\otimes{}
  &&\sum_{h \,\pmod{M}}
  e\big( - \tfrac{h \beta}{M} \big)\,
  \fraku_{\begin{psmatrix} 1 & h \\ 0 & M \end{psmatrix}}
\\
&\longmapsfrom{}&&
  \psi_M^\ast(f)
  &&{}\otimes{}
  &&\sum_{h \,\pmod{M}}
  e\big( - \tfrac{h \beta}{M} \big)\,
  \fraku_{\begin{psmatrix} 1 & h \\ 0 & 1 \end{psmatrix}}
&&{}=
  \psi_M^\ast(f) \otimes{}
  &&\sum_{h \,\pmod{M}}
  e\big( - \tfrac{N h \beta}{M} \big)\,
  \fraku_{\begin{psmatrix} 1 & N h \\ 0 & 1 \end{psmatrix}}
\\
&\lmto{}&&
  &&
  &&\sum_{h \,\pmod{M}}
  e\big(-\tfrac{N h \beta}{M} \big)\,
  (1 : N h)
&&{}={}
  &&\sum_{h \,\pmod{M}}
  e\big(-\tfrac{h \beta}{M} \big)\,
  (1 : h)
\tx{.}
\end{alignat*}
In particular, $\rmT_M(f,\beta)$ lies in the image of~\eqref{eq:thm:structure-of-ramanujan-type-congruences:hecke-to-induced}. To ease notation we will refer to the right hand side by~$\rmT_M(f,\beta)$, too.

For any $M' \isdiv M$, since the action of~$\SL{2}(\ZZ \slash M \ZZ)$ preserves congruences modulo~$M'$, we have the following inclusion of~$\OKell[ \SL{2}(\ZZ \slash M \ZZ)]$\nbd modules:
\begin{gather}
\label{eq:thm:structure-of-ramanujan-type-congruences:inclusion}
  \OKell \big\{ \bbP^1(\ZZ \slash M' \ZZ) \big\}
\lhra
  \OKell \big\{ \bbP^1(\ZZ \slash M \ZZ) \big\}
\tx{,}\quad
  (c:d)
\lmto
  \sum_{(c':d') \equiv (c:d) \,\pmod{M'}}
  (c':d')
\tx{.}
\end{gather}
Moreover, the left hand side descends to a module for~$\OKell[ \SL{2}(\ZZ \slash M' \ZZ)]$ under the projection from~$\OKell[ \SL{2}(\ZZ \slash M \ZZ)]$.

The inclusion~\eqref{eq:thm:structure-of-ramanujan-type-congruences:inclusion} yields a tensor product decomposition
\begin{gather}
\label{eq:thm:structure-of-ramanujan-type-congruences:tensor-product}
  \bigotimes_{\substack{p \tx{ prime}\\p \isdiv M}}
  \OKell \big\{ \bbP^1(\ZZ \slash M_p \ZZ) \big\}
\;\cong\;
  \OKell \big\{ \bbP^1(\ZZ \slash M \ZZ) \big\}
\end{gather}
and corresponding inclusions
\begin{gather}
\label{eq:thm:structure-of-ramanujan-type-congruences:inclusion-tensor-product}
  \bigotimes_{\substack{p \tx{ prime}\\p \isdiv M\\M_p' \isdiv M_p}}
  \OKell \big\{ \bbP^1(\ZZ \slash M_p' \ZZ) \big\}
\;\lhra\;
  \OKell \big\{ \bbP^1(\ZZ \slash M \ZZ) \big\}
\tx{,}
\end{gather}
which allow to perform representation theoretic considerations for prime powers~$M_p$ instead of~$M$. We record that the preimage of~$\rmT_M(f, \beta)$ under~\eqref{eq:thm:structure-of-ramanujan-type-congruences:tensor-product} equals
\begin{gather}
\label{eq:thm:structure-of-ramanujan-type-congruences:tensor-product-vector}
  \bigotimes_{\substack{p \tx{ prime}\\p \isdiv M}}
  \Bigg(
  \sum_{h \,\pmod{M_p}}
  e\big(-\tfrac{h \beta_p}{M} \big)\,
  (1 : h)
  \Bigg)
\tx{,}\quad
\end{gather}
where~$\beta_p := \beta 1_p$ with an integer~$1_p$ satisfying ~$1_p \equiv 1 \pmod{M_p}$ and $M_p^\# \isdiv 1_p$.

\begin{proof}[{Proof of~\eqref{eq:thm:structure-of-ramanujan-type-congruences:square-free} in Theorem~\ref{thm:structure-of-ramanujan-type-congruences}}]
For the time being assume that~$p = 2$. Our assumption that $\beta \in \ZZ \cap (M \slash 8) \ZZ_2$ ensures that $\rmT_M(f,\beta)$ lies in the image of~\eqref{eq:thm:structure-of-ramanujan-type-congruences:inclusion-tensor-product} with~$M'_2= 8$. We will show that we can assume that~$M'_2 = 2$ and thus prove~\eqref{eq:thm:structure-of-ramanujan-type-congruences:square-free}. We can and will assume that $\beta$ is odd, since we can otherwise replace~$M'_2$ by~$M'_2 \slash 2$.

Consider the domain of~\eqref{eq:thm:structure-of-ramanujan-type-congruences:inclusion-tensor-product} for~$M'_2 = 4$ and~$M'_2 = 8$. We have to inspect the tensor component for~$p = 2$ of the subrepresentation of the left hand side in~\eqref{eq:thm:structure-of-ramanujan-type-congruences:inclusion-tensor-product} generated by the preimage of~$\rmT_M(f,\beta)$. Observe that~$\SL{2}(\ZZ \slash M'_2 \ZZ)$ is a~$2$\nbd group. Since~$\ell \nisdiv 2$, we can compute with the representation~$K \{ \bbP^1(\ZZ \slash M'_2 \ZZ) \}$ instead of $\OKell \{ \bbP^1(\ZZ \slash M'_2 \ZZ) \}$. Since~$K \{ \bbP^1(\ZZ \slash M'_2 \ZZ) \}$ is a permutation representation and therefore defined over~$\QQ$, we can employ Galois automorphisms at will. For this reason, it suffices to consider the case~$\beta_2 = M \slash M'_2$.

In the case of~$M'_2 = 4$ we find through a calculation that~$\rmT_M(f,\beta + 2 M \slash M'_2)$ lies in the representation generated by~$\rmT_M(f,\beta)$. In the case of~$M'_2 = 8$ we find that the three vectors~$\rmT_M(f,\beta + 2 M \slash M'_2)$, $\rmT_M(f,\beta + 4 M \slash M'_2)$, and $\rmT_M(f,\beta + 6 M \slash M'_2)$ lie in the representation generated by~$\rmT_M(f,\beta)$. Employing Corollary~\ref{cor:congruence-on-arithmetic-progression-hecke-T-eigenvector}, we finish the proof of~\eqref{eq:thm:structure-of-ramanujan-type-congruences:square-free}.
\end{proof}

We return to the case of general~$p$. We write~$\bbF_p$ for~$\ZZ \slash p \ZZ$. We can and will assume that $\beta \in \ZZ \cap (M \slash p) \ZZ_p$, which ensures that $\rmT_M(f,\beta)$ lies in the image of~\eqref{eq:thm:structure-of-ramanujan-type-congruences:inclusion-tensor-product} with~$M'_p = p$. It suffices to investigate which $\Fell[\SL{2}(\bbF_p)]$\nbd submodule of~$\Fell \{ \bbP^1(\bbF_p) \}$ is generated by~$\sum_{h \,\pmod{p}} \zeta^h (1:h)$, where~$\zeta \in \Fell$ is the reduction modulo~$\ell$ of the $p$\thdash\ root of unity~$e( -\beta_p  \slash M)$ using the fixed embedding of~$K$ into~$\CC$.

More specifically, let~$\pi_p$ be the $p$-tensor component of the subrepresentation of the left hand side of~\eqref{eq:thm:structure-of-ramanujan-type-congruences:inclusion-tensor-product} generated by the preimage of~$\rmT_M(f,\beta)$. Observe that we have~$w \in V(\pi_p)$, where $V(\pi_p)$ denotes the representation space of~$\pi_p$.

We will use the descriptions of modular representations of~$\SL{2}(\bbF_p)$ by Bonnaf\'e in Section~9.4 of~\cite{bonnafe-2011}. We write~$\mathrm{St}_p \subset \Fell\{ \bbP^1(\bbF_p) \}$ for the Steinberg representation and $v$ for the invariant vector~$\sum_{x \in \bbP^1(\bbF_p)} x$. We record that~$w$ is not a multiple of~$v$.

\begin{proof}[{Proof of~\eqref{eq:thm:structure-of-ramanujan-type-congruences:congruence-remove-prime:Up} in Theorem~\ref{thm:structure-of-ramanujan-type-congruences}}]
If~$M_p \isdiv \beta$ and therefore~$M \isdiv \beta_p$, then the vector~$w$ does not lie in~$\mathrm{St}_p$, since we have~$\ell \nisdiv p$. By Sections~9.4.2--9.4.4 of~\cite{bonnafe-2011} it generates the representation~$\Fell\{ \bbP^1(\bbF_p) \}$. In other words, we have~$\pi_p = \Fell\{ \bbP^1(\bbF_p) \}$. In particular, we find that~$\rmT_M(f, \beta + (M \slash p) m)$ lies in the $\ell$\nbd kernel of~$\rmT_M\,\CC\,f$ for all integers~$m$. We finish the proof of~\eqref{eq:thm:structure-of-ramanujan-type-congruences:congruence-remove-prime:Up} for~$M_p \isdiv \beta$ by invoking Corollary~\ref{cor:congruence-on-arithmetic-progression-hecke-T-eigenvector} to conclude the Ramanujan-type congruence
\begin{gather*}
  \forall n \in \ZZ \,:\, c(f;\,(M \slash p) n + \beta) \equiv 0 \;\pmod{\ell}
\end{gather*}
and then iterating the argument.
\end{proof}

\begin{proof}[{Proof of~\eqref{eq:thm:structure-of-ramanujan-type-congruences:congruence-both-square-classes} in Theorem~\ref{thm:structure-of-ramanujan-type-congruences}}]
Applying~\eqref{eq:thm:structure-of-ramanujan-type-congruences:congruence-remove-prime:Up} if~$M_p \isdiv \beta$, we can and will assume that we have~$M_p \nisdiv \beta$. We therefore have~$w \in \mathrm{St}_p$, since $M \nisdiv \beta_p$ and the sum over all~$p$\thdash\ roots of unity vanishes.

Consider the case of~$\ell \nisdiv (p+1)$. If $\ell \nisdiv (p-1)$ a decomposition of~$\Fell\{ \bbP^1(\bbF_p) \}$ into irreducible components follows from the $cde$\nbd triangle, which is explained for example in Section~15 of~\cite{serre-1977}. If we have~$\ell \isdiv (p-1)$ we can use Section~9.4.2 of~\cite{bonnafe-2011}, instead. We find that~$\Fell\{ \bbP^1(\bbF_p) \}$ decomposes as~$\mathrm{St}_p \oplus \Fell v$ and~$\mathrm{St}_p$ is irreducible. As a consequence, we have~$\pi_p = \mathrm{St}_p$. In particular, we see that the vector~$\rmT_M(f, \beta + (M \slash p) m)$ lies in the $\ell$\nbd kernel of~$\rmT_M\,\CC\,f$ for all integers~$m$ such that $M_p \nisdiv (\beta + (M \slash p) m)$. We apply Corollary~\ref{cor:congruence-on-arithmetic-progression-hecke-T-eigenvector} and find the Ramanujan-type congruence with gap
\begin{gather*}
  \forall n \in \ZZ \setminus p \ZZ  \,:\,
  c(f;\, (M \slash p) n + M_p \beta') \equiv 0 \;\pmod{\ell}
\end{gather*}
with~$M_p \beta' \equiv \beta \,\pmod{M_p^\#}$. This proves~\eqref{eq:thm:structure-of-ramanujan-type-congruences:congruence-both-square-classes} if~$\ell \nisdiv (p+1)$.

Consider the case of~$\ell \isdiv (p+1)$. By our assumptions we have~$\ell \nisdiv 2$. Sections~9.4.3 and~9.4.4 of~\cite{bonnafe-2011} show that~$\Fell v$ is the socle of the length~$2$ representation~$\mathrm{St}_p$ with irreducible quotient. In particular, we find that $\pi_p = \mathrm{St}_p$. We conclude that~$\rmT_M(f, \beta')$ lies in the $\ell$\nbd kernel of~$\rmT_M\,\CC\,f$ for all integers~$\beta'$ with~$\beta' \equiv \beta \,\pmod{M_p^\#}$ and~$M_p \nisdiv \beta'$. As in the case of~$\ell \nisdiv (p + 1)$, this proves the congruence in~\eqref{eq:thm:structure-of-ramanujan-type-congruences:congruence-both-square-classes}.
\end{proof}

In the proofs of the remaining statements we continue with the setup from the proof of~\eqref{eq:thm:structure-of-ramanujan-type-congruences:congruence-remove-prime:Up} and of~\eqref{eq:thm:structure-of-ramanujan-type-congruences:congruence-both-square-classes}. In particular, we use the fact that~$\mathrm{St}_p \subseteq \pi_p$. It follows that
\begin{gather}
\label{eq:prop:ramanujan-type-implies-hecke-like:local-component}
  \sum_{h \,\pmod{p}} (1:h) - p (0:1)
\,\in\,
  V(\pi_p)
\tx{.}
\end{gather}

\begin{proof}[{Proof of~\eqref{eq:thm:structure-of-ramanujan-type-congruences:congruence-remove-prime:squarefree} in Theorem~\ref{thm:structure-of-ramanujan-type-congruences}}]
Consider the case that~$p^2 \nisdiv M_p$. For simplicity, we will formulate the proof in terms of
\begin{gather*}
  f^\#_p(\tau)
\;:=\;
  \sum_{n \in M_p^\# \ZZ + \beta}
  c(f;\,n) e(n \tau)
\tx{.}
\end{gather*}

We employ induction on~$m$ to show that~$c(f^\#_p;\, p^m n) \equiv 0 \,\pmod{\ell}$ if $p \nisdiv n$, $n \in \ZZ$. By~\eqref{eq:thm:structure-of-ramanujan-type-congruences:congruence-both-square-classes}, we have that~$c(f^\#_p;\, n ) \equiv 0 \,\pmod{\ell}$, which settles the case~$m = 0$. For the induction step, we use the relation $c( f^\#_p;\, p n' ) \equiv p^k\, c( f^\#_p;\, n' \slash p ) \,\pmod{\ell}$, which follows directly from~\eqref{eq:prop:ramanujan-type-implies-hecke-like:local-component} and~$M'_p = M_p$, and is valid for all integers~$n'$.
\end{proof}

\begin{proof}[Proof of Proposition~\ref{prop:ramanujan-type-implies-hecke-like}]
When combining the vectors in~\eqref{eq:prop:ramanujan-type-implies-hecke-like:local-component} for each~$p$, the image under the homomorphism~\eqref{eq:thm:structure-of-ramanujan-type-congruences:inclusion-tensor-product} equals
\begin{gather*}
  \bigotimes_{\substack{p \tx{ prime}\\p \isdiv M}}
  \Bigg(
  \sum_{h \,\pmod{M_p}}\hspace{-.7em}
  (1:h)
  \,-\,
  p \hspace{-.5em}
  \sum_{\substack{x \in \bbP^1(\ZZ \slash M_p \ZZ)\\x \equiv (0:1) \in \bbP^1(\bbF_p)}}\hspace{-.7em}
  x
  \,\Bigg)
\;=\;
  \bigotimes_{\substack{p \tx{ prime}\\p \isdiv M}}
  \Bigg(
  (p+1) \hspace{-.5em}
  \sum_{h \,\pmod{M_p}}\hspace{-.7em}
  (1:h)
  \,-\,
  p \hspace{-.5em}
  \sum_{x \in \bbP^1(\ZZ \slash M_p \ZZ)}\hspace{-.7em}
  x
  \,\Bigg)
\tx{.}
\end{gather*}
After taking the preimage of this under~\eqref{eq:thm:structure-of-ramanujan-type-congruences:p1-representation} and then the image under~\eqref{eq:thm:structure-of-ramanujan-type-congruences:hecke-to-induced}, the first summand yields~$(p+1) \rmU_{M_p,MN}$ and the second one equals
\begin{gather*}
  p \rmTcl_{M_p,MN}
  -
  p \rmTcl_{M_p \slash p^2, MN \slash p^2}
\tx{,}
\end{gather*}
where we use the notation in~\eqref{eq:def:classical-hecke-unnormalized} and we suppress the last summand if~$p^2 \nisdiv M_p$ as in the statement of Proposition~\ref{prop:ramanujan-type-implies-hecke-like}.
\end{proof}

\subsection{The support of Fourier expansions modulo~\texpdf{$\ell$}{l}}

The purpose of this section is to explain how to recover modular forms modulo~$\ell$ from a subset of their Fourier coefficients.

\begin{lemma}
\label{la:oldforms-modulo-ell}
Fix a number field $K \subset \CC$ and a prime ideal~$\ell \subset \OK$. Let~$f$ be a modular form of integral weight~$k$ for a Dirichlet character~$\chi$ modulo~$N$ with Fourier coefficients~$c(f;\,n) \in \OKell$. Fix a prime~$p$ that is co-prime to $\ell N$.

Assume that~$f$ satisfies the congruence
\begin{gather*}
  \forall n \in \ZZ \setminus p \ZZ \,:\, c(f;\,n) \equiv 0 \;\pmod{\ell}
\tx{.}
\end{gather*}
Then~$f$ is congruent to a constant modulo~$\ell$.

In particular, given a linearly independent set of cusp forms~$f_i \,\pmod{\ell}$ of fixed integral weight and for a fixed Dirichlet character modulo~$N$, there is a finite set of integers~$n$ co-prime to~$p$ such that the matrix with entries~$c(f_i;\,n)$ is invertible modulo~$\ell$.
\end{lemma}
\begin{proof}
The first statement is a special case of~\eqref{eq:thm:structure-of-ramanujan-type-congruences:congruence-remove-prime:squarefree} in Theorem~\ref{thm:structure-of-ramanujan-type-congruences}. Then the proof of the second part amounts to mere linear algebra.
\end{proof}

An alternative, purely geometric proof of Lemma~\ref{la:oldforms-modulo-ell} can be obtained through a generalization of Mazur's Lemma~5.9 in~\cite{mazur-1977}. A further alternative proof, which in fact establishes a significantly stronger statement, was suggested to the author by Serre. We reproduce his lemma and argument for convenience of the reader.
\begin{lemma}[Serre]
\label{la:oldforms-modulo-ell-serre}
Fix a number field $K \subset \CC$ and a prime ideal~$\ell \subset \OK$. Let~$f$ be a modular form of integral weight~$k$ for a Dirichlet character~$\chi$ modulo~$N$ with Fourier coefficients~$c(f;\,n) \in \OKell$. Assume that there is an integer~$n \ne 0$ with the property that~$c(f;\, n) \not\equiv 0 \,\pmod{\ell}$. Then given a density-zero set~$P$ of primes that are co-prime to~$\ell N$, there is~$n \in \ZZ$ with $p \nisdiv n$ for all~$p \in P$ and $c(f;\, n) \not\equiv 0 \,\pmod{\ell}$.
\end{lemma}

The proofs Lemma~\ref{la:oldforms-modulo-ell-serre} and Corollary~\ref{cor:ramanujan-type-has-density} require the next assertion that corresponds to an exercise in Serre's work. It arose from private communication with him.
\begin{remark}
\label{rm:endomorphism-set-SellN-frobenian}
For ease of notation, we set
\begin{gather*}
  V
\;:=\;
  \rmM_k(\Ga_0(N), \chi;\, \OKell) \otimes \Fell
\tx{.}
\end{gather*}
Consider an endomorphism~$\phi$ of this $\Fell$-vector space. We write~$P_\phi$ for the set of primes~$p$ that are co-prime to~$\ell N$ such that the Hecke operator~$\rmTcl_p$ acts on~$V$ as~$\phi$. Given any prime~$p$ that is co-prime to~$\ell N$, that action of the $p$\nbd th~Hecke operator is determined by~$p \,\pmod{\ell}$ and the~$p$\thdash\ Hecke eigenvalues modulo~$\ell^m$ for large enough~$m$ of the eigenforms that occur in~$\rmM_k(\Ga_0(N), \chi;\, \OKell)$. In analogy with, for instance Serre's classical argument~\cite{serre-1976}, consider the reduction modulo~$\ell^m$ of the Galois representation that is the direct sum of all~$\ell$\nbd adic Galois representations associated with Dirichlet characters modulo~$\ell N$ and with newforms that occur in the spaces~$\rmM_k(\Ga_0(N'), \chi;\, \OKell)$ for some~$N' \isdiv N$. Then the argument in~\S\,3.4.3 of~\cite{serre-2012} shows that~$P_\phi$ is $S_{\ell N}$-frobenian in the sense of~\S\,3.3.2 of~\cite{serre-2012}.
\end{remark}

\begin{proof}[Proof of Lemma~\ref{la:oldforms-modulo-ell-serre}]
Define the support modulo~$\ell$ of a modular form~$f$ with Fourier coefficients in~$\OKell$ as
\begin{gather*}
  \mathrm{supp}_\ell(f)
\;:=\;
  \big\{ n \in \QQ, n > 0 \,:\, c(f;\,n) \not\equiv 0 \,\pmod{\ell} \big\}
\tx{.}
\end{gather*}
Fix a set~$P$ of primes of density zero as in the statement, and let
\begin{gather*}
  F
\;:=\;
  \Big\{
  f \in \rmM_k(\Ga_0(N), \chi;\, \OKell) \,:\,
  \mathrm{supp}_\ell(f) \subseteq \bigcup_{p \in P} p \ZZ
  \Big\}
\tx{.}
\end{gather*}

Write~$Q$ for the set of primes~$q$ such that~$F |_{k,\chi} \rmTcl_q \subseteq F$. Observe that~$Q$ is the union of sets~$P_\phi$ as in Remark~\ref{rm:endomorphism-set-SellN-frobenian}, where~$\phi \,\pmod{\ell}$ runs through endomorphisms of the~$\Fell$-vector space $\rmM_k(\Ga_0(N),\chi;\, \OKell) \,\pmod{\ell}$ that preserve~$F$. In particular, each~$P_\phi$ and therefore~$Q$ is $S_{\ell N}$-frobenian in the sense of~\cite{serre-2012}.

Moreover, $Q$ contains every prime~$q \not\in P$ that is co-prime to~$\ell N$, and therefore has density one. By Proposition~3.8 of~\cite{serre-2012} this implies that~$Q$ equals the set of all primes not dividing~$\ell N$. In particular, we have have~$P \subset Q$.

Since~$F$ is a subset of~$\rmM_k(\Ga_0(N), \chi;\, \OKell)$, the set~$\mathrm{supp}_\ell (f)$ has a minimum for every $f \in F$ with $\mathrm{supp}_\ell (f) \ne \emptyset$. Assume by contraposition that there exists some~$f \in F$ with~$\mathrm{supp}_\ell(f) \ne \emptyset$. Then we have a well-defined minimum:
\begin{gather*}
  \min\, \mathrm{supp}_\ell (F)
\;:=\;
  \min \big\{
  \min\, \mathrm{supp}_\ell (f') \,:\,
  f' \in F,\, \mathrm{supp}_\ell (f') \ne \emptyset
  \big\}
\tx{.}
\end{gather*}
We can and will assume that~$f \in F$ is chosen in such a way that
\begin{gather*}
  n_f := \min\, \mathrm{supp}_\ell (f) = \min\, \mathrm{supp}_\ell (F)
\tx{.}
\end{gather*}
By the definition of~$F$ there is~$p \in P$ such that~$p \isdiv n_f$. We have~$n_f \slash p \in \mathrm{supp}_\ell(f |_{k,\chi}\, \rmTcl_p)$, contradicting the choice of~$f$. In particular, all elements of~$F$ are congruent to constants modulo~$\ell$.

Now let~$f$ be as in the statement. By the assumptions, $f$ is not congruent to a constant modulo~$\ell$, and therefore~$f \not\in F$. Then the existence of~$n$ as in the lemma follows from the definition of~$F$.
\end{proof}

\section{Ramanujan-type and Hecke congruences}
\label{sec:ramanujan-type-hecke}

Ramanujan-type congruences with gap~\eqref{eq:thm:structure-of-ramanujan-type-congruences:congruence-both-square-classes} in Theorem~\ref{thm:structure-of-ramanujan-type-congruences} are one of the key discoveries in this paper. They provide a link between Ramanujan-type congruences and congruences for Hecke eigenvalues. In this section, we will decipher this relation.

For completeness, we refer to Example~\ref{ex:hecke-operators:oldforms-twists} for the construction of Ramanujan-type congruences at the level.

\subsection{Ramanujan-type congruences constructed through the~\texpdf{$\Theta$}{Theta}-operator}
\label{sec:ramanujan-type-hecke:theta-operator}

Prior to embarking on the study of congruences for Hecke eigenvalues if~$\ell \nisdiv M$, we point out Ra\-ma\-nu\-jan-type congruences that can be constructed through the~$\Theta$\nbd operator and the Hecke operator~$\rmU_\ell$ modulo~$\ell$. Correspondingly, our first proposition is centered around the case that~$\ell$ divides~$M$. We emphasize in our statement that the Ra\-ma\-nu\-jan-type congruences that we obtain are maximal.
\begin{proposition}
\label{prop:integral-weight-ell-congruences}
Let~$\ell$ be an odd prime, $m$ positive integers, and~$\beta$ an integer. Then there is a modular form~$f$ of even weight and level~$1$ such that
\begin{alignat*}{2}
&
  \forall n \in \ZZ \,:\,
  c(f;\,\ell^m n + \ell^{m-1} \beta &)
&\equiv
  0 \;\pmod{\ell}
\quad\tx{and}
\\
&
  \exists n \in \ZZ \,:\,
  c(f;\,\ell^{m-1} n &)
&\not\equiv
  0 \;\pmod{\ell}
\tx{.}
\end{alignat*}
\end{proposition}

\begin{remark}
In order to extend Proposition~\ref{prop:integral-weight-ell-congruences} to arbitrary levels, one can use the following construction. Let~$f$ be a modular form satisfying the Ramanujan-type congruence
\begin{alignat*}{2}
&
  \forall n \in \ZZ \,:\,
  c(f;\,& M n + \beta)
&\equiv
  0 \;\pmod{\ell}
\quad\tx{and}
\\
  \forall q \tx{ prime}, q \isdiv M\;
&
  \exists n \in \ZZ \,:\,
  c(f;\,& (M \slash q) n + \beta )
&\not\equiv
  0 \;\pmod{\ell}
\tx{.}
\end{alignat*}
Let~$p \nisdiv M$ be prime and~$\chi$ the trivial Dirichlet character modulo~$p$. For a nonnegative integer~$m$, we set~$f_{\chi\,p^m}(\tau) := f_\chi(p^m \tau)$, where~$f_\chi \not\equiv 0 \,\pmod{\ell}$ is as in Example~\ref{ex:hecke-operators:oldforms-twists}. Assume that $m > 0$, let~$f_0 := f_{\chi\,p^{m-1}}$ and~$f_1 := f_{\chi\,p^{m-1}} - f_{p^{m-1}}$, and choose integers~$\beta_0$ and~$\beta_1$ with
\begin{gather*}
  \beta_0,\beta_1 \equiv p^{m-1} \beta \,\pmod{M}
\tx{,}\quad
  \beta_0 \equiv 0 \,\pmod{p^m}
\tx{,}\quad
  \beta_1 \equiv 0 \,\pmod{p^{m-1}}
\tx{,}\;
  \beta_1 \not\equiv 0 \,\pmod{p^m}
\tx{.}
\end{gather*}
Then for the following Ramanujan-type congruences hold:
\begin{alignat*}{2}
&
  \forall n \in \ZZ \,:\,
  c( f_0 ;\,& M p^m n + \beta_0 )
&\equiv
  0 \;\pmod{\ell}
\quad\tx{and}
\\
  \forall q \tx{ prime}, q \isdiv M p^m \;
&
  \exists n \in \ZZ \,:\,
  c(f_0;\,& (M p^m \slash q) n + \beta_0 )
&\not\equiv
  0 \;\pmod{\ell}
\tx{;}
\\
&
  \forall n \in \ZZ \,:\,
  c( f_1 ;\,& M p^m n + \beta_1 )
&\equiv
  0 \;\pmod{\ell}
\quad\tx{and}
\\
  \forall q \tx{ prime}, q \isdiv M p^m \;
&
  \exists n \in \ZZ \,:\,
  c(f_1;\,& (M p^m \slash q) n + \beta_1 )
&\not\equiv
  0 \;\pmod{\ell}
\tx{.}
\end{alignat*}
\end{remark}

\begin{proof}[{Proof of Proposition~\ref{prop:integral-weight-ell-congruences}}]
In this proof all modular forms have even weight and rational and~$\ell$\nbd integral Fourier coefficients. For~$\beta' \in \ZZ$, we set
\begin{gather}
  \rmU_{\ell,\beta'}\, g
\;:=\;
  \sum_{n \equiv \beta' \,\pmod{\ell}} c(g;n) e(n \tau)
\end{gather}
and~$\rmU_\ell := \rmU_{\ell,0}$. Recall the operator~$\Theta := (2 \pi i)^{-1}\, \partial_\tau$ and its action on Fourier coefficients.

The assertion in the proposition is equivalent to finding a modular form~$f$ level~$1$ such that~$\rmU_{\ell,\beta}\, \rmU_\ell^{m-1}\, f \equiv 0 \pmod{\ell}$ and~$\rmU_\ell^{m-1}\, f \not\equiv 0 \pmod{\ell}$.

Consider the case~$\ell \isdiv \beta$. Observe that there is a modular form~$g$ such that~$\rmU_{\ell,\beta'}\, g \not\equiv 0 \,\pmod{\ell}$ for some~$\ell \nisdiv \beta'$, as can be seen when iterating the~$\rmU_\ell$ operator. Then we set~$g_1 \,:=\, \Theta\, g \not\equiv 0 \,\pmod{\ell}$ and find that~$\rmU_\ell\, \Theta\, g \equiv 0 \,\pmod{\ell}$. If~$m = 1$, we set~$f = g_1$ and finish the proof.

Consider the case~$\ell \nisdiv \beta$. There is a modular form~$g$ such that~$\rmU_\ell\, g \not\equiv 0 \,\pmod{\ell}$, since the~$\rmU_\ell$\nbd operator modulo~$\ell$ is surjective~\cite{dewar-2012}. We set
\begin{gather*}
  g_1
\;:=\;
  g - \left(\mfrac{\beta}{\ell}\right) \Theta^{\frac{\ell-1}{2}}\, g
\not\equiv
  0
  \;\pmod{\ell}
\tx{.}
\end{gather*}
Inspecting Fourier coefficients, we discover that~$\rmU_{\ell,\beta}\, g_1 \equiv 0 \,\pmod{\ell}$. If~$m = 1$, we set~$f = g_1$ and complete the proof.

For~$m > 1$, we employ the surjectivity of the~$\rmU_\ell$ operator. Let~$g_1$ be as before and let~$g_m$ be a level~$1$ modular form such that~$\rmU_\ell^{m-1}\,g_m \equiv g_1 \,\pmod{\ell}$. Then we finish the proof with~$f = g_m$.
\end{proof}

\begin{example}
In specific cases, we can make Proposition~\ref{prop:integral-weight-ell-congruences} more explicit, but the resulting modular forms quickly have gargantuan weight. For instance, choose~$g = \Delta$, $\ell = 17$, and~$\beta$ co-prime to~$\ell$ and not a square modulo~$\ell$. Then the~$g_2$ in the proof of Proposition~\ref{prop:integral-weight-ell-congruences} has filtration weight~$3180$ and its Fourier expansion is of the form
\begin{multline*}
  7 e(\tau) +  9 e(2 \tau) + e(17 \tau) + 10 e(34 \tau)
  +
  7 e(68 \tau) + 7 e(136 \tau) + 2 e(153 \tau)
\\
  +
  7 e(221 \tau) + 11 e(255 \tau)
  +
  \cO(e(266 \tau))
\tx{.}
\end{multline*}
\end{example}

\subsection{Ramanujan-type congruences implied by Hecke congruences}
\label{sec:ramanujan-type-hecke:implied-by-hecke}

We start our investigation of how congruences for Hecke eigenvalues imply Ramanujan-type congruences with a special case of a theorem by Treneer~\cite{treneer-2006}. The second part of the next proposition was not stated by Treneer, but readily follows from her proof.
\begin{proposition}
[{cf.\@ Theorem~1.2~(ii) by Treneer~\cite{treneer-2006}}]
\label{prop:integral-weight-p-congruences-treneer}
Let~$\ell$ be an odd prime. Fix a modular form~$f$ of integral weight for a Dirichlet character modulo~$N$. Assume that~$f$ has~$\ell$\nbd integral Fourier coefficients and that~$\ell \nisdiv N$.

Given a non-negative integer~$m$, let~$P_m$ be the set of prime~$p \equiv -1 \,\pmod{\ell N}$ with the property that we have the Ramanujan-type congruence with gaps
\begin{gather*}
  \forall n \in \ZZ \setminus (p \ZZ \cup \ell \ZZ) \,:\,
  c(f;\,\ell^m p n)
\equiv
  0 \;\pmod{\ell}
\tx{.}
\end{gather*}
Then there is an~$m$ such that~$P_m$ has positive upper density.

If~$f$ is a cusp form, the set of primes~$p \equiv -1 \,\pmod{\ell N}$ with the property that
\begin{gather*}
  \forall n \in \ZZ \setminus p \ZZ \,:\,
  c(f;\,p n)
\equiv
  0 \;\pmod{\ell}
\end{gather*}
has positive upper density.
\end{proposition}
\begin{remark}
By~\eqref{eq:thm:structure-of-ramanujan-type-congruences:congruence-remove-prime:Up} in Theorem~\ref{thm:structure-of-ramanujan-type-congruences}, we see that the congruences in Proposition~\ref{prop:integral-weight-p-congruences-treneer} cannot be improved if~$\rmU_\ell^m\, f \not\equiv 0 \,\pmod{\ell}$ in the first case or~$f \not\equiv 0 \,\pmod{\ell}$ in the second one.
\end{remark}

The key idea in the proof of this proposition goes back to work of Ono and Ahlgren-Ono~\cite{ono-2000,ahlgren-ono-2001}, who used a density statement for Hecke eigenvalues by Serre~\cite{serre-1976} in conjunction with the recursion for Fourier coefficients implied by a congruence for Hecke eigenvalues. Serre's statement naturally comes in two flavors, and Treneer showcased the consequences of one of them. The next proposition captures all Ra\-ma\-nu\-jan-type congruences that are implied by congruences for Hecke eigenvalues, and allows us to make use of the other flavor of Serre's theorem. It features the classical Hecke operator, which we in this paper denote by~$\rmTcl_p$.
\begin{proposition}
\label{prop:hecke-implies-ramanujan-type}
Fix a number field $K \subset \CC$ and a prime ideal~$\ell \subset \OK$ with~$\ell \nisdiv 2$. Let~$f$ be a modular form of integral weight~$k$ for a Dirichlet character~$\chi$ modulo~$N$ with Fourier coefficients~$c(f;\,n) \in \OKell$. Consider a prime~$p$ that is co-prime to~$\ell N$. Assume that we have the congruence
\begin{gather*}
  f \big|_{k,\chi} \rmTcl_p
\;\equiv\;
  \lambda_p f
  \;\pmod{\ell}
\tx{,}\quad
  \lambda_p \in \OKell
\tx{.}
\end{gather*}

If\/~$\lambda_p^2 \equiv 4 \chi(p) p^{k-1} \,\pmod{\ell}$, we have the Ramanujan-type congruence with gaps
\begin{gather}
\label{eq:prop:hecke-congruence-gives-ramanujan-type-congruence:congruence}
  \forall n \in \ZZ \setminus p \ZZ \,:\,
  c(f;\,p^m n)
\equiv
  0 \;\pmod{\ell}
\end{gather}
for a positive integer~$m$ if and only if~$m \equiv -1 \,\pmod{\ell}$. If~$\lambda_p^2 \not\equiv 4 \chi(p) p^{k-1} \,\pmod{\ell}$ and the~$\rmL$\nbd polynomial
\begin{gather*}
  1 - \lambda_p X + \chi(p) p^{k-1} X^2
\;\equiv\;
  (1 - \alpha_p X)(1 - \beta_p X )
  \;\pmod{\ell}
\end{gather*}
factors over~$\Fell$, then~\eqref{eq:prop:hecke-congruence-gives-ramanujan-type-congruence:congruence} holds if and only if\/~$\alpha_p^{m+1} \equiv \beta_p^{m+1} \,\pmod{\ell}$.
\end{proposition}
\begin{proof}
We capture the Fourier expansion modulo~$\ell$ of~$f$ as a formal Dirichlet series
\begin{gather*}
  \sum_{n = 1}^\infty
  c(f;\,n) n^{-s}
\;\equiv\;
  \frac{1}{1 - \lambda_p p^{-s} + \chi(p) p^{k-1} p^{-2s}}
  \sum_{\substack{n = 1\\p\nisdiv n}}^\infty
  c(f;\,n) n^{-s}
\tx{.}
\end{gather*}
After extending~$K$ if needed, we can factor the $\rmL$\nbd polynomial as in the statement. We then recognize that the congruence
\begin{gather*}
  \sum_{i = 0}^m \alpha_p^i \beta_p^{m-i}
\;\equiv\;
  0
  \;\pmod{\ell}
\end{gather*}
is equivalent to the Ramanujan-type congruence with gaps in~\eqref{eq:prop:hecke-congruence-gives-ramanujan-type-congruence:congruence}.

If $\lambda_p^2 \equiv 4 \chi(p) p^{k-1} \,\pmod{\ell}$, we have~$\alpha_p \equiv \beta_p \,\pmod{\ell}$, and~\eqref{eq:prop:hecke-congruence-gives-ramanujan-type-congruence:congruence} occurs for~$\ell \isdiv (m+1)$, i.e., $m \equiv -1 \,\pmod{\ell}$ as claimed. 

Consider the case~$\lambda_p^2 \not\equiv 4 \chi(p) p^{k-1} \,\pmod{\ell}$, which implies~$\alpha_p \not\equiv \beta_p \,\pmod{\ell}$. Since we have~$\ell \nisdiv p$, $p \nisdiv N$, and~$\alpha_p \beta_p = \chi(p) p^{k-1}$, both~$\alpha_p$ and~$\beta_p$ are units modulo~$\ell$. We can therefore rephrase the congruence condition as
\begin{gather*}
  \sum_{i = 0}^m \alpha_p^i \beta_p^{m-i}
\;=\;
  \frac{\alpha_p^{m+1} - \beta_p^{m+1}}
       {\alpha_p - \beta_p}
\;\equiv\;
  0
  \;\pmod{\ell}
\tx{.}
\end{gather*}
\end{proof}

Combining Proposition~\ref{prop:hecke-implies-ramanujan-type} with Serre's results, we obtain the following congruences.
\begin{corollary}
\label{cor:hecke-implies-ramanujan-type:serre}
Let~$\ell$ be an odd prime. Fix a modular form~$f$ of integral weight for a Dirichlet character modulo~$N$. Assume that~$f$ has~$\ell$\nbd integral Fourier coefficients and that~$\ell \nisdiv N$.

Given a non-negative integer~$m$, let~$P_m$ be the set of prime~$p \equiv 1 \,\pmod{\ell N}$ with the property that we have the Ramanujan-type congruence with gaps
\begin{gather*}
  \forall n \in \ZZ \setminus (p \ZZ \cup \ell \ZZ) \,:\,
  c(f;\,\ell^m p^{\ell-1} n)
\equiv
  0 \;\pmod{\ell}
\tx{.}
\end{gather*}
Then there is an~$m$ such that~$P_m$ has positive upper density.

If~$f$ is a cusp form, the set of primes~$p \equiv 1 \,\pmod{\ell N}$ with the property that
\begin{gather*}
  \forall n \in \ZZ \setminus p \ZZ \,:\,
  c(f;\,p^{\ell-1} n)
\equiv
  0 \;\pmod{\ell}
\end{gather*}
has positive upper density.
\end{corollary}
\begin{proof}
We can reduce our considerations to the case of cusp forms by invoking Theorem~3.1 of~\cite{treneer-2006}. Serre states in Exercise~6.4 of~\cite{serre-1976} (see also Corollary~3.13 of~\cite{serre-2012}) that a set of positive upper density of primes~$p \equiv 1 \,\pmod{\ell N}$ has the property that $\rmTcl_p$ acts by multiplication with~$2$ on the space of cusp forms modulo~$\ell$ of fixed integral weight for fixed Dirichlet character. In particular, there is a set of positive upper density of such primes with~$f |_k \rmTcl_p \equiv 2 f \,\pmod{\ell}$. We can now employ Proposition~\ref{prop:hecke-implies-ramanujan-type} with~$\lambda_p = 2$ and~$\lambda_p^2 \equiv 4 \equiv 4 \chi(p) p^{1-k} \,\pmod{\ell}$.
\end{proof}

Serre's result arises from a combination of the Chebotarev Density Theorem and integral models for $\ell$-adic Galois representations associated with modular forms. The latter encode the Hecke eigenvalues of newforms via their image under the (conjugacy classes of the) Frobenius elements. The former then yields a positive proportion of Frobenius elements in any conjugacy class of the image modulo~$\ell$ once the intersection of that conjugacy class with the image modulo~$\ell$ is not empty. There are only two elements of the absolute Galois group of~$\QQ$ whose image can be uniformly described for all newforms, namely the identity and the complex conjugation. Their images correspond to the two cases in Serre's theorem.

\subsection{Hecke modules}
\label{sec:ramanujan-type-hecke:hecke-modules}

While Hecke operators~$\rmTcl_p$ act semi-simply on spaces of modular forms~$\rmM_k(\Ga_0(N), \chi)$, $p \nisdiv N$, there are plenty of exceptions to this for modular forms modulo~$\ell$, which arise from congruences between cusp forms~\cite{ribet-1983}. Concrete examples from the early days of the theory were provided by, for example, Hida~\cite{hida-1981} and Zagier~\cite{zagier-1985}. An investigation of the situation in level~$1$, formulated in more modern terms, was pursued by~Bella\"{i}che and Khare~\cite{bellaiche-khare-2015}, while the case of Eisenstein congruences for square-free levels, rooted in work of Mazur~\cite{mazur-1977} on prime levels, appears in work of Wake and Wang-Erickson~\cite{wake-wang-erickson-2021}.

\begin{example}
When restricting to cusp forms of level~$N = 1$, the first examples occur already in weight~$k = 24$. For example, if~$\ell = 3$ and~$p \in \{2, 5\}$, then~$\rmTcl_p$ acts nilpotently, but not as zero. Specifically, we have a generalized Hecke eigenforms
\begin{gather*}
  f_1
\;\equiv\;
  e(2 \tau) + 2 e(5 \tau) + \cO(e(6 \tau))
  \,\pmod{3}
\tx{,}\quad
  f_0
\;\equiv\;
  2 e(\tau) + \cO(e(3 \tau))
  \,\pmod{3}
\end{gather*}
with~$\rmTcl_2\,f_1 \equiv f_0 \,\pmod{3}$ and~$\rmTcl_2\, f_0 \equiv 0 \,\pmod{3}$.

One of the first examples with~$\lambda \not\equiv 0 \,\pmod{\ell}$ arises from weight~$k = 24$, $\ell = 7$, and~$p = 2$. We then have generalized Hecke eigenforms
\begin{align*}
  f_1
\;&\equiv\;
  e(\tau) + 4 e(2 \tau) + 5 e(4 \tau) + \cO(e(6 \tau))
  \,\pmod{7}
\tx{,}
\\
  f_0
\;&\equiv\;
  5 e(\tau) + 5 e(2 \tau) + \cO(e(3 \tau))
  \,\pmod{7}
\end{align*}
with~$(\rmTcl_2 - 4)\,f_1 \equiv f_0 \,\pmod{7}$ and~$(\rmTcl_2 - 4)\, f_0 \equiv 0 \,\pmod{7}$. Observe that also this case is slightly special, since~$\lambda^2 \equiv 4 \chi(p) p^{k-1} \,\pmod{\ell}$, so that the $\rmL$-polynomial of~$f_0$ is not separable modulo~$\ell$.
\end{example}

Since we cannot find a basis of Hecke eigenforms of~$\rmM_k(\chi; \OKell) \,\pmod{\ell}$ in all cases, Proposition~\ref{prop:hecke-implies-ramanujan-type} does not strictly exhaust all possibilities. In this section, we show that Ramanujan-type congruences remain stable under the action of Hecke operators. In particular, Ramanujan-type congruences for some modular form imply Ramanujan-type congruences for Hecke eigenforms modulo~$\ell$.

Fix a number field $K \subset \CC$, a prime ideal~$\ell \subset \OK$, an integral weight~$k$, and a Dirichlet character~$\chi$ modulo~$N$. Consider a positive integer~$M$, $\gcd(M, \ell N) = 1$, and an integer~$\beta$. We let~$M_{\rms\rmf}$ be the largest square-free divisor of~$M$, and assume that~$M$ exactly divides~$\beta M_{\rms\rmf}$. We define the~$\OKell$\nbd modules
\begin{gather}
\begin{aligned}
\label{eq:ramanujan-type-congruence-modules}
  \rmR^\times_\ell\big( M \ZZ + \beta,\, \rmS_k(\chi) \big)
\;&:=\;
  \big\{
  \rmS_k(\chi; \OKell) \,:\,
  \forall n \in \ZZ,\, \gcd(n,M) = 1 \,:\,
\\&
\hphantom{:=\; \big\{ \rmS_k(\chi; \OKell) \,:\,}
\qquad\qquad
  c(f; (M \slash M_{\rms\rmf}) n) \equiv 0 \,\pmod{\ell}
  \big\}
\\
  \rmR_\ell\big( M \ZZ + \beta,\, \rmS_k(\chi) \big)
\;&:=\;
  \big\{
  \rmS_k(\chi; \OKell) \,:\,
  \forall n \in \ZZ \,:\, c(f; M n + \beta) \equiv 0 \,\pmod{\ell}
  \big\}
\tx{.}
\end{aligned}
\end{gather}
The condition that~$M$ exactly divides~$\beta M_{\rms\rmf}$ is required to conveniently formulate the first set. This condition does not restrict our scope by Theorem~\ref{thm:structure-of-ramanujan-type-congruences}~\ref{it:thm:structure-of-ramanujan-type-congruences:square-free}.

The next theorem states that the $\OKell$\nbd modules in~\eqref{eq:ramanujan-type-congruence-modules} are Hecke-modules away from~$\gcd(M, \ell N)$ for odd~$\ell$. If~$\ell$ is even only Ramanujan-type congruences with gap are stable under the Hecke action for integral weights.
\begin{theorem}
\label{thm:hecke-module-ramanujan-type-congruences}
Consider the sets of cusp forms with a Ramanujan-type congruence (with gap) as in~\eqref{eq:ramanujan-type-congruence-modules}. Given a prime~$p$, $p \nisdiv \gcd(M, \ell N)$, we have
\begin{gather*}
  \rmR^\times_\ell\big( M \ZZ + \beta,\, \rmS_k(\chi) \big)
  \big|_{k,\chi}\, \rmTcl_p
\;\subseteq\;	
  \rmR^\times_\ell\big( M \ZZ + \beta,\, \rmS_k(\chi) \big)
\tx{.}
\end{gather*}
If~$\ell \nisdiv 2$, we have
\begin{gather*}
  \rmR_\ell\big( M \ZZ + \beta,\, \rmS_k(\chi) \big)
  \big|_{k,\chi}\, \rmTcl_p
\;\subseteq\;
  \rmR_\ell\big( M \ZZ + \beta,\, \rmS_k(\chi) \big)
\tx{.}
\end{gather*}
\end{theorem}
The proof of Theorem~\ref{thm:hecke-module-ramanujan-type-congruences} requires some preparation, which will equally come to use in the proof of Theorem~\ref{thm:ramanujan-type-implies-hecke}.

We write~$\rmL(f,s) = \sum_{n = 1}^\infty c(f;n) n^{-s}$ for the $\rmL$\nbd series associated with a modular form~$f$ for, say, $\Gamma_1(N)$, and~$\rmL^p(f,s) = \sum_{p \nisdiv n} c(f;n) n^{-s}$ for the $\rmL$\nbd series away from~$p$. Given a generalized Hecke eigenform~$f$ of eigenvalue~$\lambda_p \,\pmod{\ell}$ under~$\rmTcl_p$, we let the~$\rmL$\nbd polynomial at~$p$ be~$\rmL_p(\lambda_p, X) := 1 - \lambda_p X + \chi(p) p^{k-1} X^2$. Observe that we suppress~$k$ and~$\chi$ from our notation. In complete analogy with modular forms, we will say that two~$\rmL$-series are congruent modulo~$\ell$, if their coefficients are.

Consider an eigenform~$f$ of the classical Hecke operator~$\rmTcl_p$ with eigenvalue~$\lambda_p$. The~$\rmL$\nbd series of~$f$ admits the factorization~$\rmL(f,s) = \rmL_p(\lambda_p,p^{-s})^{-1}\, \rmL^p(f,s)$. The next lemma extends this to generalized Hecke eigenforms.
\begin{lemma}
\label{la:generalized-hecke-eigenform}
Fix a number field $K \subset \CC$ and a prime ideal~$\ell \subset \OK$. Let~$f$ be a modular form of integral weight~$k$ for a Dirichlet character~$\chi$ modulo~$N$ with Fourier coefficients~$c(f;\,n) \in \OKell$. Consider a prime~$p$ that is co-prime to~$\ell N$. Assume that we have the generalized congruence
\begin{gather*}
  f \big|_{k,\chi} \big( \rmTcl_p - \lambda_p \big)^{d+1}
\;\equiv\;
  0
  \;\pmod{\ell}
\tx{,}\quad
  \lambda_p \in \OKell
\tx{,}
\end{gather*}
for some nonnegative integer~$d$. Given a nonnegative integer~$t$, set
\begin{gather*}
  f_t
\;:=\;
  f \big|_{k,\chi} \big( \rmTcl_p - \lambda_p \big)^t
\tx{.}
\end{gather*}

Then we have
\begin{gather}
\label{eq:la:generalized-hecke-eigenform}
  \rmL(f,s)
\;\equiv\;
  \sum_{t = 0}^d \frac{p^{-s t}\, \rmL^p(f_t,s)}{\rmL_p(\lambda_p,p^{-s})^{t+1}}
  \;\pmod{\ell}
\tx{.}
\end{gather}
\end{lemma}
\begin{proof}
For any modular form~$g \in \rmM_k(\Ga_0(N), \chi)$ and any integer~$n$, we have the usual formula
\begin{gather}
\label{eq:la:generalized-hecke-eigenform:fourier-coefficient-hecke-formula}
  c\big( g \big|_{k,\chi}\, \rmTcl_p ; n \big)
\;=\;
  \left\{
  \begin{alignedat}{2}
  & c(g; np)
  \tx{,}\quad
  && \tx{if $p \nisdiv n$;}
  \\
  & c(g; np) + \chi(p) p^{k-1} c(g; n \slash p)
  \tx{,}\quad
  && \tx{if $p \isdiv n$.} 
  \end{alignedat}
  \right.
\end{gather}
Applying this to $f_t$, $0 \le t < d$, we find the recursion formula
\begin{gather*}
  \rmL_p(\lambda_p, p^{-s})\, \rmL(f_t, s)
  \,-\,
  \rmL^p(f_t)
\;\equiv\;
  p^{-s}\,
  \rmL(f_{t+1})
  \;\pmod{\ell}
\tx{,}
\end{gather*}
where the second term on the left hand side arises from the condition that~$n$ is an integer in the previous coefficient formula. After solving for~$L(f_t,s)$ this yields the statement when applying induction on~$d$.
\end{proof}

\begin{proof}%
[{Proof of Theorem~\ref{thm:hecke-module-ramanujan-type-congruences}}]
Theorem~\ref{thm:structure-of-ramanujan-type-congruences}~\ref{it:thm:structure-of-ramanujan-type-congruences:ramanujan-type-congruence-with-gap} says that if~$\ell \nisdiv 2$, Ramanujan-type congruences are equivalent to Ramanujan-type congruences with gap. Therefore the second part of the Theorem~\ref{thm:hecke-module-ramanujan-type-congruences} follows from the first one.

Fix an element~$f$ of~$\rmR^\times_\ell( M \ZZ + \beta,\, \rmS_k(\chi) )$. If we have~$p \nisdiv M$, the Fourier coefficient formula alike the on in~\eqref{eq:la:generalized-hecke-eigenform:fourier-coefficient-hecke-formula} implies that~$f | \rmTcl_p \in \rmR^\times_\ell( M \ZZ + \beta,\, \rmS_k(\chi) )$ as desired. Therefore, we can and will assume that~$p \isdiv M$ and hence~$\gcd(p, \ell N) = 1$.

We have a Ramanujan-type congruence modulo~$\ell$ of~$f$ on~$M \ZZ + \beta'$ for all integer~$\beta'$ that satisfy~$\gcd(\beta',M) = \gcd(\beta,M)$. We factor~$M = M_p M_p^\#$ with a~$p$\nbd power~$M_p$ and~$M_p^\#$ co-prime to~$p$, and set
\begin{gather*}
  f_p
\;:=\;
  \sum_{\substack{n \in \ZZ\\\gcd(n,M_p^\#) = \gcd(\beta, M_p^\#)}}
  c(f;\,n) e(n\tau)
\;=\;
  f \big|_{k,\chi}\,
  \prod_{q^{m+1} \isexdiv M_p^\#}
  \rmU_q^m \big( 1 - \rmU_q \rmV_q \big) \rmV_q^m
\tx{.}
\end{gather*}
From the second expression for~$f_p$ is becomes clear that it is a modular forms for the Dirichlet character~$\chi$ modulo~$M_p^{\#\,2} N$. From the defining expression for~$f_p$, we deduce that the Ramanujan-type congruence with gap for~$f$ is equivalent to
\begin{gather*}
  \forall n \in \ZZ \setminus p \ZZ \,:\,
  c(f_p;\,(M_p \slash p) n)
\equiv
  0 \;\pmod{\ell}
\tx{.}
\end{gather*}
Since~$\rmTcl_p$ commutes with both~$\rmU_q$ and~$\rmV_q$, $q \ne p$, in the remainder of the proof we can and will assume that~$M = M_p = p^{m+1}$, $m$ a nonnegative integer.

By Lemma~\ref{la:generalized-hecke-eigenform}, it suffices to show that for every nonnegative integer~$t$, we have
\begin{gather*}
  \forall n \in \ZZ \setminus p \ZZ \,:\,
  c(f_t;\,(M_p \slash p) n)
\equiv
  0 \;\pmod{\ell}
\tx{,}
\end{gather*}
where in the remainder of the proof~$f_t$ is as in Lemma~\ref{la:generalized-hecke-eigenform}. We also adopt~$d$ from Lemma~\ref{la:generalized-hecke-eigenform}. We write~$c(g(X); m)$ for the $m$\thdash\ coefficient of a power series~$g(X)$. By inspecting the coefficient of~$(p^m n)^{-s}$ on the right hand side of~\eqref{eq:la:generalized-hecke-eigenform}, for any positive integer~$n$ co-prime to~$p$, we obtain that
\begin{gather*}
  c(f; p^m n)
\;\equiv\;
  \sum_{t = 0}^{d}
  c\big( X^t\, \rmL_p(\lambda_p, X)^{-t-1};\, m \big)\,
  c\big( f_t; n \big)
  \;\pmod{\ell}
\tx{.}
\end{gather*}

We can and will assume that~$d$ is minimal, so that the modular forms~$f_t$, $0 \le t \le d$, are linearly independent modulo~$\ell$. By Lemma~\ref{la:oldforms-modulo-ell} there is a set of positive integers~$n$ co-prime to~$p$ such that the matrix with entries~$c(f_t;\,n)$, $0 \le t \le d$, is invertible modulo~$\ell$. In particular, we conclude the congruence~$c\big( X^t\, \rmL_p(\lambda_p, X)^{-t-1},\, m \big) \equiv 0 \,\pmod{\ell}$ for all~$0 \le t \le d$.

Applying Lemma~\ref{la:generalized-hecke-eigenform} again, we find that, for~$0 \le j \le d$,
\begin{gather*}
  c(f_j; p^m n)
\;\equiv\;
  \sum_{t = 0}^{d-j}
  c\big( X^t\, \rmL_p(\lambda_p, X)^{-t-1};\, m \big)\,
  c\big( f_{t+j}; n \big)
  \;\pmod{\ell}
\tx{.}
\end{gather*}
By this relation, the congruences~$c\big( X^t\, \rmL_p(\lambda_p, X)^{-t-1},\, m \big) \equiv 0 \,\pmod{\ell}$ imply the desired congruences~$c(f_j; p^m n) \equiv 0 \,\pmod{\ell}$.
\end{proof}

\subsection{Hecke congruences implied by Ramanujan-type congruences}
\label{sec:ramanujan-type-hecke:hecke-congruence}

In this section, we will show that under the hypothesis that~$\ell \nisdiv 2$ Ramanujan-type congruences are equivalent to specific congruences for Hecke eigenvalues. Theorem~\ref{thm:hecke-module-ramanujan-type-congruences} allows us to restrict to generalized Hecke eigenforms.
\begin{theorem}
\label{thm:ramanujan-type-implies-hecke}
Fix a number field $K \subset \CC$ and a prime ideal~$\ell \subset \OK$, $\ell \nisdiv 2$. Let~$f$ be a modular form of integral weight~$k$ for a Dirichlet character~$\chi$ modulo~$N$ with Fourier coefficients~$c(f;\,n) \in \OKell$. Fix a power~$M$ of a prime~$p$, $\gcd(p, \ell N) = 1$, and an integer~$\beta$. Assume that~$f$ is a generalized Hecke eigenform modulo~$\ell$ for~$\rmTcl_p$ of eigenvalue~$\lambda_p$ of depth~$d$.

Provided that~$f \not\equiv 0 \,\pmod{\ell}$, the following are equivalent:
\begin{enumeratearabic}
\item
\label{it:thm:ramanujan-type-implies-hecke:ramanujan}
We have the Ramanujan-type congruence
\begin{gather*}
  \forall n \in \ZZ \,:\,
  c(f;\,M n + \beta) \equiv 0 \;\pmod{\ell}
\tx{.}
\end{gather*}

\item
\label{it:thm:ramanujan-type-implies-hecke:hecke}
For every\/~$0 \le t \le d$ the coefficient of exponent\/~$\gcd(M, \beta)$ in the following formal power series vanishes modulo~$\ell$:
\begin{gather*}
  X^t\, \big( 1 - \lambda_p X + \chi(p) p^{k-1} X^2 \big)^{-t-1}
\tx{.}
\end{gather*}
\end{enumeratearabic}
\end{theorem}
\begin{remark}
Congruences for the coefficients of~$( 1 - \lambda_p p^{-s} + \chi(p) p^{k-1} p^{-2s} )^{-1}$ characterize Ramanujan-type congruences on~$M \ZZ + \beta$ of the Hecke eigenform associated with~$f$. In particular, Proposition~\ref{prop:hecke-implies-ramanujan-type} yields a more explicit description of them.
\end{remark}

Before proving Theorem~\ref{thm:ramanujan-type-implies-hecke}, we establish one corollary on densities of Ra\-ma\-nu\-jan-type congruences.
\begin{corollary}
\label{cor:ramanujan-type-has-density}
Let $\ell$, $f$, $M$, and~$\beta$ be as in Theorem~\ref{thm:ramanujan-type-implies-hecke}, and fix a positive integer~$m$. Then the set of primes~$p$ such that
\begin{gather*}
  \exists \beta \in \ZZ \,:\,
  \forall n \in \ZZ \,:\, c(f;\,M p^m n + \beta) \equiv 0 \;\pmod{\ell}
\end{gather*}
is frobenian in the sense of~\S\,3.3.2 of Serre~\cite{serre-2012}. In particular, it has a density.
\end{corollary}
\begin{proof}
Combine the characterization in Theorem~\ref{thm:ramanujan-type-implies-hecke} with Remark~\ref{rm:endomorphism-set-SellN-frobenian}.
\end{proof}

\begin{remark}
In a similar vein, by virtue of Propositions~\ref{prop:hecke-implies-ramanujan-type}, Theorem~\ref{thm:hecke-module-ramanujan-type-congruences}, and Remark~\ref{rm:endomorphism-set-SellN-frobenian}, any cusp form as in Theorem~\ref{thm:ramanujan-type-implies-hecke} that has a Ramanujan-type congruence modulo~$\ell$ on, say, the arithmetic progression~$p^m (p \ZZ + \beta_p)$, $\ell \nisdiv p$, has Ramanujan-type congruences modulo~$\ell$ on $q^m (q \ZZ + \beta_q)$ for a set of primes~$q$ with~$q \equiv p \,\pmod{\ell}$ of positive density.
\end{remark}

\begin{proof}[{Proof of Theorem~\ref{thm:ramanujan-type-implies-hecke}}]
Since we assume that~$\ell \nisdiv 2$, by Theorem~\ref{thm:structure-of-ramanujan-type-congruences}~\ref{it:thm:structure-of-ramanujan-type-congruences:ramanujan-type-congruence-with-gap} and~\ref{it:thm:structure-of-ramanujan-type-congruences:beta-equiv-zero}, we can assume that~$M$ exactly divides~$\beta p$. We write~$M = p^{m + 1}$ for a nonnegative integer~$m$.

The proof follows closely the lines of the proof of Theorem~\ref{thm:hecke-module-ramanujan-type-congruences}. Adopting notation from there, for any positive integer~$n$ co-prime to~$p$, we have
\begin{gather*}
  c(f; p^m n)
\;\equiv\;
  \sum_{t = 0}^{d}
  c\big( X^t\, \rmL_p(\lambda_p, X)^{-t-1};\, m \big)\,
  c\big( f_t; n \big)
  \;\pmod{\ell}
\tx{.}
\end{gather*}
This reveals that~\ref{it:thm:ramanujan-type-implies-hecke:hecke} implies~\ref{it:thm:ramanujan-type-implies-hecke:ramanujan}.

For the converse, we employ Lemma~\ref{la:oldforms-modulo-ell}. It guarantees that there is a set of positive integers~$n$ co-prime to~$p$ such that the matrix with entries~$c(f_t;\,n)$, $0 \le t \le d$, is invertible modulo~$\ell$. This suffices to establish that~\ref{it:thm:ramanujan-type-implies-hecke:ramanujan} implies~\ref{it:thm:ramanujan-type-implies-hecke:hecke}.
\end{proof}

We finish this paper with a proof of Corollary~\ref{maincor:no-ramanujan-type-congruences} and a justification of Example~\ref{mainex:delta-function}.

\begin{proof}[{Proof of Corollary~\ref{maincor:no-ramanujan-type-congruences}}]
By Theorems~\ref{thm:hecke-module-ramanujan-type-congruences} and~\ref{thm:ramanujan-type-implies-hecke}, we can employ the conditions listed in Proposition~\ref{prop:hecke-implies-ramanujan-type} to rule out Ramanujan-type congruences. Using notation from Proposition~\ref{prop:hecke-implies-ramanujan-type}, in the case of~$\lambda_p^2 \equiv 4 \chi(p) p^{k-1} \,\pmod{\ell}$ the assumption that~$\ell \nisdiv (m+1)$ excludes such congruences. In the case $\lambda_p^2 \not\equiv 4 \chi(p) p^{k-1} \,\pmod{\ell}$, we have the condition~$\alpha_p^{m+1} \equiv \beta_p^{m+1} \,\pmod{\ell}$. Since~$m+1$ is co-prime to the order~$\ell-1$ of the unit group~$\bbF_\ell^\times = (\ZZ \slash \ell \ZZ)^\times$, this is equivalent to~$\alpha_p \equiv \beta_p \,\pmod{\ell}$. This entails the congruences~$4 \lambda_p^2 \equiv 4 \chi(p) p^{k-1} \,\pmod{\ell}$, a contradiction.
\end{proof}

\begin{proof}[{Proof of Example~\ref{mainex:delta-function}}]
Assume that we have a Ramanujan-type congruence modulo~$\ell$ for~$\Delta$ on the arithmetic progression~$M \ZZ + \beta$. We first show that we may assume that~$M$ is a prime power. Consider the case that it is not. Then there is a prime~$p \ne \ell$ with~$p \isdiv M$ and~$M_p^\# \ne 1$, where we use the same notation as in Theorem~\ref{mainthm:ramanujan-type-implies-hecke}. 

We have to show that~$\Delta$ has a Ramanujan-type congruence on at least one of the progressions~$M_p^\# \ZZ + \beta$ or~$M_p \ZZ + \beta$. Consider the modular form
\begin{gather*}
  \Delta_p(\tau)
\;:=\;
  \sum_{\substack{n \in \ZZ\\\gcd(n,M_p^\#) = \gcd(\beta,M_p^\#)}}
  c(\Delta;\,n) e(n \tau)
\tx{,}
\end{gather*}
which has a Ramanujan-type congruence on~$M_p \ZZ + \beta$ by Theorem~\ref{thm:structure-of-ramanujan-type-congruences}~\ref{it:thm:structure-of-ramanujan-type-congruences:ramanujan-type-congruence-with-gap}. We can and will assume that $\Delta_p \not\equiv 0 \,\pmod{\ell}$, since we have otherwise discovered a Ramanujan-type congruence for~$\Delta$ on~$M_p^\# \ZZ + \beta$.

We apply Theorem~\ref{mainthm:ramanujan-type-implies-hecke} to find that the Ramanujan-type congruence for~$\Delta_p$ arises from congruences for Hecke eigenvalues. Since~$\Delta$ and hence~$\Delta_p$ is an eigenform for the~$p$\thdash\ classical Hecke operator, the following congruence suffices:
\begin{gather*}
  \Delta_p \big|_{12}\, \rmTcl_p
\;\equiv\;
  c(\Delta, p) \Delta_p
  \;\pmod{\ell}
\tx{.}
\end{gather*}
We now apply Proposition~\ref{prop:hecke-implies-ramanujan-type} to show that~$\Delta$ already has a Ramanujan-type congruence on~$M_p \ZZ + \beta$.

The case that~$M$ is a power of~$\ell$ requires us to prove that we may assume that~$M = \ell$. If we have a~$\rmU_\ell$-congruence this is clear. We can therefore assume that the $\rmU_\ell$\nbd eigenvalue of~$\Delta \,\pmod{\ell}$ does not vanish. We can employ Proposition~\ref{prop:congruences-square-free-quotient} to reduce ourselves to the case that~$M = \ell^{m+1}$ and~$\beta = \ell^m \beta_0$ for some~$\beta_0$ with~$\ell \nisdiv \beta_0$. To conclude the argument, we use that~$\Delta \,\pmod{\ell}$ is an eigenform with nonzero eigenvalue under~$\rmU_\ell^m$. This implies that we have a Ramanujan-type congruence on the arithmetic progression~$\ell \ZZ + \beta_0$.
\end{proof}

\renewbibmacro{in:}{}
\renewcommand{\bibfont}{\normalfont\small\raggedright}
\renewcommand{\baselinestretch}{.8}

\Needspace*{4em}
\printbibliography[heading=none]%

\Needspace*{3\baselineskip}
\noindent
\rule{\textwidth}{0.15em}

{\noindent\small
Chalmers tekniska högskola och G\"oteborgs Universitet,
Institutionen för Matematiska vetenskaper,
SE-412 96 Göteborg, Sweden\\
E-mail: \url{martin@raum-brothers.eu}\\
Homepage: \url{http://raum-brothers.eu/martin}
}%

\end{document}

